\documentclass[12pt,a4paper]{amsart}

\usepackage{amsmath}
\usepackage{amsfonts}
\usepackage{amssymb}
\usepackage{amsthm}
\usepackage{calrsfs}

\usepackage{color}

\newtheorem{theorem}{Theorem}[section]
\newtheorem{lemma}[theorem]{Lemma}
\newtheorem{proposition}[theorem]{Proposition}
\newtheorem{remark}[theorem]{Remark}
\newtheorem*{acknowledgments}{Acknowledgments}

\newcommand{\R}{\mathbb{R}} 
\newcommand{\C}{\mathbb{C}} 
\newcommand{\N}{\mathbb{N}} 

\newcommand{\duality}[2]{\left\langle #1 \, ,\, #2 \right\rangle} 
\newcommand{\norm}[3]{\left\|#1\right\|^{#2}_{#3}} 

\DeclareMathOperator{\supp}{supp} 
\DeclareMathOperator{\dist}{dist} 

\newcommand{\ov}[1]{\overline{#1}}
\newcommand{\p}{\partial}

\author[]{Pedro Caro \and Valter Pohjola}
\title[]{Stability estimates for an inverse problem for the magnetic Schr\"odinger operator}
\date{January 14, 2013}
\address{Department of Mathematics and Statistics, Helsingin yliopisto / Helsingfors universite / University of Helsinki, Finland}
\email{pedro.caro@helsinki.fi}
\email{valter.pohjola@helsinki.fi}

\begin{document}

\begin{abstract} In this paper we prove stable determination of an inverse boundary value problem associated to a magnetic Schr\"odinger operator assuming that the magnetic and electric potentials are essentially bounded and the magnetic potentials admit a H\"older-type modulus of continuity in the sense of $L^2$.
\end{abstract}

\maketitle

\tableofcontents
\setcounter{tocdepth}{1}

\section{Introduction and main results} \label{sec:intro}
Let $ U $ be a bounded non-empty open subset of $ \R^n $ (from now on a domain)
with $ n \geq 3 $. Given a magnetic potential $A$ and an electric potential $q$, in the Lebesgue
spaces 
$L^\infty (U; \C^n)$ and $L^\infty (U)$ respectively,
we consider the magnetic Schr\"odinger operator  formally given by
\[
    L_{A,q} u = - (\nabla + i A) \cdot (\nabla + i A) u + q u.
\]
This corresponds to the operator
$ L^U_{A, q} :H^1(U) \longrightarrow H^{-1} (U) $ given by
\[
    \duality{L^U_{A, q} u}{v} = \int_{U} \nabla u \cdot \nabla v + i
    A \cdot (u \nabla v - v \nabla u) + (A \cdot A + q) u v \, dx
\]
for any $ u \in H^1(U) $ and $ v \in H^1_0(U) $. Here $H^1(U)$ denotes
the first order Sobolev space based in the Lebesgue space $L^2(U)$. The space
$H^1_0(U)$ denotes the closure in $H^1(U)$, of the compactly supported smooth functions in
$U$ and $H^{-1}(U)$ denotes its dual. For convenience, $ A \cdot A $ will be
denoted by $ A^2 $ and we will write $ L_{A,q} $ instead of $ L^U_{A,q} $
whenever the domain associated to this definition is clear. Note also that $ L_{A,
q} $ is linear and bounded.

Next we describe the boundary data of a $ H^1(U) $ solution $ u $ to the magnetic Schr\"odinger equation 
\begin{equation*}
    L_{A,q} u = 0
\end{equation*}
and then we define the Cauchy data set associated to this equation. It is well
known that the trace space of $ H^1(U) $, denoted here by $ TH^1(U) $, is
described by the quotient $ H^1(U)/H_0^1(U) $. The space $ TH^1(U) $ endowed
with the quotient norm, denoted by $ \norm{\cdot}{}{TH^1(U)} $, is a
Banach space. The trace map 
$ T_U : H^1(U) \longrightarrow TH^1(U) $ is defined by $ T_U u = [u] $ for any
$ u \in H^1(U) $, where $ [u] $ denotes the equivalence class of $ u $. For
convenience, we will write $ T $ instead of $ T_U $ whenever the associated
domain is clear. 

The normal component of the magnetic gradient on the boundary is, in a regular enough 
setting, given by $(\p_{\nu}+i\nu\cdot A)u|_{\p U}$
(where $\nu$ denotes the outward pointing unit normal vector on the boundary of $U$ denoted by $\partial U$). 
In our case we define this following \cite{KU}, as the bounded linear map 
$N^U_{A, q} : H^1(U) \longrightarrow TH^1(U)^\ast $ given by
\[ 
    \duality{N^U_{A,q} u}{g} = \int_{U} \nabla u \cdot \nabla v + i A
\cdot (u \nabla v - v \nabla u) + (A^2 + q) u v \, dx 
\]
for any $ u \in H^1(U) $ such that $ L_{A, q}u = 0 $ and any $ g \in
TH^1(U) $ such that $ g = Tv $. The space $ (TH^1(U)^\ast,
\norm{\cdot}{}{TH^1(U)^\ast}) $ denotes the dual space of $ (TH^1(U),
\norm{\cdot}{}{TH^1(U)}) $. Again, we will write $ N_{A,q} $ instead of 
$N^U_{A,q} $ whenever the domain is clear. 
Finally, the Cauchy data set of $ H^1(U)
$ solutions to the magnetic Schr\"odinger equation is defined as
\[
    C_{A, q} = \{ (Tu, N_{A, q} u) : u \in H^1(U),\, L_{A, q} u = 0 \}. 
\]
From now on, $ C_{A, q} $ will be referred as the Cauchy data set associated to
the operator $ L_{A,q} $. Note that $C_{A, q}$ encodes the information of the
solutions on the boundary of $U$, hence it is usually called boundary
measurements.

The inverse boundary value problem (IBVP for short) considered in this paper consists in recovering the magnetic and electric potentials from the knowledge of their associated Cauchy data set. Related to this problem, some other natural questions arise as uniqueness and stability.

The first question can be stated as follows: Given two magnetic
potentials $A_1,A_2 \in L^\infty (U; \C^n)$ and two
electric potentials $q_1, q_2 \in L^\infty (U)$, does 
$C_{A_1, q_1} = C_{A_2, q_2}$ imply $A_1 = A_2$ and
$q_1 = q_2$? The answer to this question is negative 
because of the 
following obstruction. For every $\varphi$ in the space\footnote{
The space $ \mathrm{Lip}(1,\ov{U})$ is the 
space of Lipschitz continuous functions, see \cite{St} for the precise
definition.}
$\mathrm{Lip}(1,\ov{U})$ 
with $\varphi (x) = 0$ for all $x \in \partial U$, one has
$C_{A+\nabla \varphi, q} = C_{A, q}$ (see \cite{KU} for details). Thus, from
the boundary measurements one can not distinguish between $A$ and $A + \nabla \varphi$. 
This problem does however not effect the magnetic field $dA$, which is interpreted
as follows.
Recall that any vector field $A \in L^\infty (U; \C^n)$ with components $A_j$
can be identified with the $1$-form
\[ \sum_{j = 1}^n A_j dx_j, \]
still denoted by $A$. The magnetic field induced by the potential $A$ 
is now given by
\[ 
    dA = \sum_{1 \leq j < k \leq n} (\partial_{x_j} A_k - \partial_{x_k} A_j) d_{x_j} \wedge d_{x_k}. 
\]
Due to the lack of smoothness of $A$, this definition has to be understood in
the sense of currents (i.e. differential forms in the sense of distributions). 
The magnetic potentials $A$ and $A + \nabla \varphi$ induce the same magnetic field,
since $d(A+d \varphi) = dA$ (where $d\varphi = \sum_1^n \partial_{x_j} \varphi dx_j$).
The non uniqueness described above does therefore not
extend to the magnetic fields. Thus, the problem we will consider consists in recovering the magnetic
field $dA$ and the electric potential $q$ from the Cauchy data set.

The question of stability essentially ask whether it is possible to provide a
quantitative answer to the (qualitative) question of uniqueness. More
precisely, if the proximity of the magnetic fields and electric potentials can
be estimated by the proximity of their corresponding Cauchy data sets. In order
to study  the question of stability, one should have some notion of proximity
between Cauchy data sets. Let $ A_1, A_2 \in L^\infty(U; \C^n) $ be two
magnetic potentials and let $ q_1, q_2 \in L^\infty(U) $ be two electric
potentials. Consider $ C_{A_j, q_j} $ the Cauchy data set associated to $
L_{A_j, q_j} $ with $ j \in \{ 1, 2 \} $. Given $ (f_j, g_j) \in C_{A_j, q_j} $
with $ j \in \{ 1, 2 \} $ set
\[ 
    I \big((f_j, g_j ); C_{A_k, q_k}\big) = \inf_{(f_k, g_k) \in C_{A_k, q_k}} \left[
    \norm{f_j - f_k}{}{TH^1(U)} + \norm{g_j - g_k}{}{TH^1(U)^\ast} \right], 
\]
with $ k \in \{ 1, 2 \} $. We define the pseudo-metric distance between $
C_{A_1, q_1} $ and $ C_{A_2, q_2} $ as
\begin{equation*}
\dist (C_{A_1, q_1}, C_{A_2, q_2}) = \max_{j, k \in \{ 1, 2 \}} \sup_{\substack{(f_j, g_j) \in C_{A_j, q_j} \\ \norm{f_j}{}{TH^1(U)} = 1}} I ((f_j, g_j); C_{A_k, q_k}).
\end{equation*}
This notion of proximity was introduced in \cite{Ca} and it has been successfully used to study the stability of certain IBVP on frameworks where the forward problem is ill-posed (see \cite{Ca11} and \cite{INUW}).

The uniqueness and the stability of this IBVP have been studied by several
authors under various regularity assumptions on the magnetic and electric
potentials. In \cite{Su}, a local uniqueness result was established for
magnetic potentials in $W^{2,\infty}$ and $L^\infty$ electric potentials --the
local nature of the result is due to a smallness condition imposed to the
magnetic potential. In \cite{NaSuU}, the smallness condition was removed for
smooth magnetic and electric potentials, and for compactly supported $C^2$
magnetic potentials and $L^\infty$ electric potentials.  The uniqueness results
were subsequently extended to $C^1$ magnetic potentials in \cite{T}, to some
less regular but small potentials in \cite{P}, and to Dini continuous magnetic
potentials in \cite{Sa}. The best result by now is \cite{KU}, by Krupchyk and
Uhlmann, where they proved uniqueness assuming the magnetic and electric
potentials to be essentially bounded. Furthermore, they do not require
regularity assumptions for the boundary of the domain. 
{Uniqueness for the closely related inverse scattering problem with a 
magnetic potential has been studied by Eskin and Ralston in \cite{ER}.}

The question of
stability has been studied in \cite{Tz} by Tzou. There, a $\log$-type stability
estimate is established for the IBVP studied in this paper,
assuming that the boundary of the domain
is smooth, the magnetic potentials are in $W^{2, \infty}$ 
with equal values on the boundary
and the electric potentials are in $L^\infty$.

The questions of uniqueness, stability and reconstruction for non-smooth frameworks have been recently studied for several IBVP as the Calder\'on problem (see \cite{AsP}, \cite{ClFR} and \cite{FRo} for dimension $n = 2$ and \cite{BTo}, \cite{HT}, \cite{CaGR} and \cite{GZh} for $n \geq 3$) and for an IBVP associated to the time-harmonic Maxwell equations (see \cite{CaZ}).

In this paper, we consider the question of stability associated to the previously described IBVP.
We improve considerably the stability result by Tzou
providing a quantitative version of the result proved by Krupchyk and Uhlmann.
In order to state precisely our result, we need to introduce some notation.

Given a domain $\Omega$ in $\R^n$ and two constants $M \in [ 1, + \infty) $ and $\varepsilon \in (0, 1)$, we define the class of admissible magnetic potentials, denoted by $\mathcal{A} (\Omega, M,\varepsilon, r)$ with $r \in [1, +\infty)$ or $r = \infty$, as the class of $A \in L^\infty (\Omega; \C^n)$ such that its extension by zero out of $\Omega$, still denoted by $A$, satisfies the a priori bound
\[ \norm{A}{}{L^\infty (\R^n; \C^n)} + |A|_{B^{2, r}_\varepsilon} \leq M. \]
Here
\[ |A|^2_{B^{2, r}_\varepsilon} = \sum_{j = 1}^n \left( \int_{\R^n} \frac{\norm{A_{j}(\cdot + y) - A_{j}}{r}{L^2(\R^n)}}{|y|^{n + r\varepsilon}} \, dy \right)^{2/r} \]
for $r \in [1, +\infty)$ and
\[ |A|^2_{B^{2, \infty}_\varepsilon} = \sum_{j = 1}^n \sup_{y \in \R^n \setminus \{ 0 \}} \frac{\norm{A_j(\cdot + y) - A_j}{2}{L^2(\R^n)}}{|y|^{2\varepsilon}} \]
with $ A_j $ denoting $ j $-th component of $ A $. Note that if $\partial \Omega$ can be locally described by the graph of a Lipschitz function and
\[ \norm{A}{}{L^\infty (\Omega; \C^n)} + |A|_{B^{2, r}_\varepsilon (\Omega)} \leq M, \]
then the extension by zero of $A$ out of $\Omega$ will satisfies (see \cite{Tr}) an a priori bound depending on $M$ as well as $n$ and $\Omega$. The same should happen for more general boundaries. This has been studied in \cite{FRo} for the case of Sobolev spaces $W^{s,p} (\Omega)$.

For $\Omega$ and $M$ as above, we also define the class of admissible electric potentials $\mathcal{Q} (\Omega, M)$ as the class of $q \in L^\infty (\Omega)$ such that
\[ \norm{q}{}{L^\infty (\Omega)} < M. \]
The extension by zero out of $\Omega$ of $q\in \mathcal{Q} (\Omega, M)$ will be also denoted by $q$.

\begin{theorem} \label{th:magnetic} \sl Let $\Omega$ be a domain in $\R^n$ and consider two constants $M \in [ 1, + \infty) $ and $\varepsilon \in (0, 1)$. There exists a constant $c_0 \in (0, 1)$, depending on $M$ and $\Omega$ such that if $A_1, A_2 \in \mathcal{A} (\Omega, M,\varepsilon, r)$ with $r \in [1, +\infty)$ or $r = \infty$, $q_1, q_2 \in \mathcal{Q} (\Omega, M)$, $C_j$ denotes the Cauchy data set associated to $A_j, q_j$ and $|\log \dist(C_1, C_2)|^{-1} < c_0$, then
\[
    \|  d A_1- d A_2  \|_{H^{-1}\boldsymbol\Omega^2(\R^n)}  
    \lesssim 
	\big|  \log \dist(C_1,C_2) \big|^{-c\varepsilon^2 / n}
\]
with $c \in (0, 1)$ universal. Moreover, if $\delta \in (1 - \varepsilon, 1)$ then
\[
    \|  d A_1- d A_2  \|_{B^{2, r}_{-\delta} \boldsymbol\Omega^2(\R^n)}  
    \lesssim 
	\big|  \log \dist(C_1,C_2) \big|^{-c \varepsilon (\delta - 1 + \varepsilon) / n}.
\]
The implicit constant in these estimates depend on $M$ and $\varepsilon$ as well as on $n$ and $\Omega$. The implicit constant on the second one also depends on $\delta$.
\end{theorem}

The symbol $ \lesssim $ holds for $ \leq $ modulo a multiplicative constant.
This constant is called here implicit constant. On the other hand, if $X(G)$
with $G$ a non-empty open {subset} of $\R^n$ denotes a function space, then $X
\boldsymbol \Omega^k (G)$ denotes the corresponding space for differential
forms of degree $k$. In particular, the definitions of
$H^{-1}\boldsymbol\Omega^2(\R^n)$ and $B^{2, r}_{-\delta}
\boldsymbol\Omega^2(\R^n)$ can be found right before Proposition \ref{dAStab}
and Proposition \ref{prop:dAbesov} in Section \ref{sec:stabilityMAG},
respectively.

\begin{theorem} \label{th:electric} \sl Consider $\lambda \in (0, 1]$ and $\theta \in (0, 2/n)$. Under the same assumptions as in Theorem \ref{th:magnetic} we have that there exists a constant $c_0 \in (0, 1)$ depending on $M$, $\Omega$, $n$, $\theta$ and $\varepsilon$ such that if $|\log \dist(C_1, C_2)|^{-1} < c_0$, then
\[
    \|  q_1- q_2  \|_{H^{-\lambda}(\R^n)}  
    \lesssim 
	\big|  \log \dist(C_1,C_2) \big|^{- c \theta \varepsilon^3 \lambda / n^2} 
\]
with $c \in (0, 1)$ universal.
Moreover, if $q_j \in B^{2, r}_\varepsilon (\R^n)$ and satisfies the a priori bound $\norm{q_j}{}{B^{2, r}_\varepsilon (\R^n)} < M$ then
\[
    \|  q_1- q_2  \|_{B^{2, r}_0(\R^n)}  
    \lesssim 
	\big|  \log \dist(C_1,C_2) \big|^{- c\theta \varepsilon^4 / n^2}.
\]
The implicit constant in these estimates depend on $M, \varepsilon$ and $\theta$ as well as on $n$ and $\Omega$. The implicit constant on the first one also depends on $\lambda$.
\end{theorem}

The definition of the spaces $H^{-\lambda}(\R^n)$ and $B^{2, r}_0(\R^n)$ can be
found right before Proposition \ref{dqStab} and Proposition \ref{prop:qSTAB} in
Section \ref{sec:stabilityELEC}, respectively. On the other hand, $f \in B^{2,
r}_\varepsilon (\R^n)$ if $f \in L^2(\R^n)$ and $|f|_{B^{2, r}_\varepsilon} <
\infty$. Note that $|\cdot|_{B^{2, r}_\varepsilon}$ has been defined for
vector fields, the definition for functions is similar.

Theorem \ref{th:magnetic} is consequence of Proposition \ref{dAStab} and
Proposition \ref{prop:dAbesov}. The stability estimates there are stated for norms with negative index since $dA_j$ has to be understood in a weak sense. On the other hand, Theorem \ref{th:electric} is consequence of Proposition \ref{dqStab} and Proposition \ref{prop:qSTAB}. The first stability estimate is stated for the norm of $H^{-\lambda} (\R^n)$ but, assuming an a priori upper bound on the norm $H^{\varepsilon} (\R^n)$ of the potentials, one could deduce a stability estimate controlling the $L^2(\R^n)$-norm of the potentials. We should not expect to prove a stability estimate controlling the $L^\infty(\R^n)$-norm since the potentials are not continuous.

Regarding the second {parts of the estimates} in Theorem \ref{th:magnetic} and Theorem
\ref{th:electric}, it is worth to point out that whenever $r = 2$ the norm of
the spaces $B^{2, r}_{-\delta} \boldsymbol\Omega^2(\R^n)$ and $H^{-\delta}
\boldsymbol\Omega^2(\R^n)$ and the spaces $B^{2, r}_0(\R^n)$ and $L^2(\R^n)$
are equivalent respectively.
Thus, the second estimates in the theorems
generalize the ones we would get by interpolation between the first estimates
in the theorems and the corresponding a priori bounds.

Let us now explain the main difficulties  and ideas in the proofs of Theorem \ref{th:magnetic} and Theorem
\ref{th:electric}. We start by recalling the qualitative argument due to
Krupchyk and Uhlmann. Their starting point is the following integral identity
\[
    \int_{\Omega} i (A_1 - A_2) \cdot (u_1 \nabla \overline{u_2} -
    \overline{u_2} \nabla u_1) + (A_1^2 - A_2^2 + q_1 - q_2) u_1 \overline{u_2} \,
    dx = 0, 
\]
which holds for $u_1$ and $u_2$ solving $ L_{A_1, q_1} u_1 = 0 $ and $
L_{\overline{A_2}, \overline{q_2}} u_2 = 0 $ respectively, whenever  $C_{A_1,q_1} = C_{A_2,q_2}$. 
They then proceed by constructing so called 
complex geometric optics solutions (CGOs for short) that are to be used with the
integral identity.
The CGOs are solutions of the form
\[ u = e^{\zeta \cdot x/h} (a + r(h)) \]
where $\zeta$ is a complex vector, $h$ is a small parameter, $a$ is a sort of
complex amplitude and $r(h)$ is a correction term that vanishes when $h$ goes
to zero. 
With the CGOs
and the integral identity at hand, they deduce that
$dA_1 = dA_2$. The next step for them was to prove that $q_1 = q_2$.
Using  the fact that $dA_1 = dA_2$ is unfortunately not by itself enough to remove
the $A_1$ and $A_2$ terms from the integral identity and  isolate 
the term containing $q_1 - q_2$. They solved this problem by
using the Poincar\'e lemma for currents to conclude that $A_1 = A_2 + \nabla
\varphi$, since $dA_1 = dA_2$. 
This allowed them to consider the pair of potentials $(A_1,q_1)$ and $(A_1 - \nabla \varphi,q_2)$, 
instead of the original ones.
Then, they exploited the gauge invariance of the Cauchy data sets in a
ball $B$ containing $\ov{\Omega}$, by picking a $\varphi$ that
vanishes on the boundary $\partial B$, to conclude that 
$C_{A_1, q_1} = C_{A_1 - \nabla \varphi,q_2}$, and hence that $C_{A_1, q_1} = C_{A_1,q_2}$.
Thus they could assume that $A_1=A_2$ in the above integral identity
and they could isolate the term containing $q_1 - q_2$ to prove that $q_1 = q_2 $.

Krupchyk and Uhlmann's construction of CGOs is based on the use of Carleman
estimates, and its main feature is that they only need to make approximation of
the magnetic potentials by smooth vector fields in the $L^2$
sense.\footnote{ Recently Haberman and Tataru proved in  \cite{HT} uniqueness for
the Calder\'on problem with continuously differentiable conductivities. The
reason why their argument does not provide uniqueness for general Lipschitz
conductivities is because, in the construction of the CGOs, they required to
approximate the gradient of conductivities in $L^\infty$ sense.}
Regarding a quantitative counterpart of Krupchyk and Uhlmann's approach, the
first point will be to find an appropriate class of magnetic potentials for
which the rate of approximation by smooth vector fields in the $L^2$ sense
(with respect to $h$) is the same. To do this, we only need to
prescribe an $L^2$ modulus of continuity and define the class as all the
magnetic potentials admitting this modulus of continuity. However, in order to
obtain the optimal stability for this IBVP, namely $\log$ type, we need to
assume that this modulus of continuity is of H\"older type, say of order
$\varepsilon$. This suggests examining magnetic potentials in the Besov spaces $B^{2,
r}_\varepsilon$. With this choice one can then relatively straight forwardly
prove stability for the magnetic
fields using the following integral estimate
\begin{align*}
\bigg| \int_{\Omega} i (A_1 - A_2) &\cdot (u_1 \nabla \overline{u_2} -
    \overline{u_2} \nabla u_1) + (A_1^2 - A_2^2 + q_1 - q_2) u_1 \overline{u_2}
    \, dx \bigg| \\
& \lesssim \dist(C_1, C_2) \norm{u_1}{}{H^1(\Omega)} \norm{u_2}{}{H^1(\Omega)}.
\end{align*}
The most difficult step is to prove stability for the electric potentials. One
is again faced with 
the problem of isolating the term containing $q_1 - q_2$, only controlling 
the difference of the magnetic fields $dA_1-dA_2$. 
A natural idea is then to mimic the uniqueness proof, use the gauge
invariance of the Cauchy data sets in the ball $B$ to modify the integral
estimate above and plug in appropriate CGOs. More precisely, use $B$ instead of $\Omega$, replace $A_2$ by $A_2
+ \nabla \varphi$, for a $\varphi$ in\footnote{ The space
$W^{1, p} (B)$ is the first-order Sobolev
space based on $L^p$ with $p \in [1, + \infty)$ or $p = \infty$.}
$W^{1, \infty} (B)$ with $\varphi|_{\partial B} = 0$, and plug in CGOs for $L^B_{A_1, q_1} u_1 = 0$ and $L^B_{\ov{A_2 + \nabla
\varphi}, \ov{q_2}} u_2 = 0$. 
The crucial point here is that the $A_1 - A_2$ term, that we cannot hope to control due
to the non uniqueness of the magnetic potentials, is replaced by 
$A_1 - (A_2 + \nabla \varphi)$ in the integral estimate. This later term can be controlled by the difference $dA_1- dA_2$.
One does this by choosing $\varphi$ suitably so that one is able to derive the estimate
\begin{equation}
    \norm{A_1 - (A_2 + \nabla \varphi)}{}{L^2(B; \C^n)} \lesssim \norm{dA_1 -
    dA_2}{}{H^{-1} \boldsymbol \Omega^2 (B)}
    \label{es:de_la_intro}
\end{equation}
where $\varphi \in W^{1, \infty} (B)$ for which $\varphi|_{\partial B} =
0$. 
An appropriate choice for $\varphi$ is the exact component of the Hodge decomposition 
$A_1-A_2$, which vanishes
on $\partial B$. It should be mentioned here that this is also roughly the idea in
\cite{Tz}, which deals with the case of more regular potentials. This idea needs however
several modifications to work in the less regular framework. The main reason
for the need to carry out these modifications is that the estimate we are able
to prove only
holds for $\varphi$ in 
$W^{1,p} (B)$ for every $p \geq n$. The restriction $p \neq \infty$ is consequence of
the elliptic regularity, which only holds for $1<p<+\infty$. Thus, if we did
not modify the previous approach, we could not use Krupchyk and Uhlmann's
method to construct CGOs for $L^B_{\ov{A_2 + \nabla \varphi}, \ov{q_2}} u_2 =
0$ since $\varphi \notin W^{1, \infty} (B)$. Finally, let us point out that
proving \eqref{es:de_la_intro} becomes in our case more technical than in
\cite{Tz} due to to the lack of regularity. The argument in \cite{Tz} is based on the
open mapping theorem and it is enough to prove the
bijectivity of certain operator --which is a qualitative property. Our
approach is however based on the $H^1$ ellipticity of the Hodge Laplacian and a
compactness argument.

The paper is organized as follows. In Section \ref{sec:integralID} we prove the
integral estimate that will be used as the starting point of our argument. In
Section \ref{sec:CGOs} we review the construction of the CGOs due to Krupchyk
and Uhlmann for the special case where the magnetic potentials satisfy a
prescribed $L^2$ modulus of continuity of H\"older type. In Section
\ref{sec:stabilityMAG} and Section \ref{sec:stabilityELEC} we prove stability
for the magnetic fields and the electric potentials respectively. In Section
\ref{sec:keyPOINT} we prove estimate \eqref{es:de_la_intro}, which is the key
ingredient in the proof of the stability for the electric potentials.

\section{From the boundary to the interior} \label{sec:integralID}
In this section we prove an integral estimate relating the electric and
magnetic potentials in $ \Omega $ with the distance between their corresponding
Cauchy data sets. This integral estimate will be our starting point in proving
the stability estimates for the IBVP under consideration.

\begin{proposition} \label{prop:integralID} \sl Let $ A_1, A_2 \in L^\infty(\Omega; \C^n) $ be two magnetic potentials and let $ q_1, q_2 \in L^\infty(\Omega) $ be two electric potentials. Let $ C_j $ with $ j \in \{ 1, 2 \} $ denote the Cauchy data set associated to the operator $ L_{A_j, q_j} $. Then, for any $ u_1 \in H^1(\Omega) $ solving $ L_{A_1, q_1} u_1 = 0 $ and any $ u_2 \in H^1(\Omega) $ solving $ L_{\overline{A_2}, \overline{q_2}} u_2 = 0 $, we have that
\begin{gather*}
\left| \int_{\Omega} i (A_1 - A_2) \cdot (u_1 \nabla \overline{u_2} - \overline{u_2} \nabla u_1) + (A_1^2 - A_2^2 + q_1 - q_2) u_1 \overline{u_2} \, dx \right| \\ \lesssim \dist(C_1, C_2) \left[ 1 + \norm{A_2}{2}{L^\infty(\Omega; \C^n)} + \norm{q_2}{}{L^\infty(\Omega)} \right] \norm{u_1}{}{H^1(\Omega)} \norm{u_2}{}{H^1(\Omega)},
\end{gather*}
where the implicit constant is universal.
\end{proposition}
\begin{proof} Note that
\begin{gather} \label{id:startBOUNDARY}
\duality{N^U_{A_1, q_1} u_1}{\overline{T_U u_2}} - \overline{\duality{N^U_{\overline{A_2}, \overline{q_2}} u_2}{\overline{T_U u_1}}} \\
= \int_{U} i (A_1 - A_2) \cdot (u_1 \nabla \overline{u_2} - \overline{u_2} \nabla u_1) + (A_1^2 - A_2^2 + q_1 - q_2) u_1 \overline{u_2} \, dx \nonumber
\end{gather}
for any domain $ U \subset \R^n $. In this proof we only use the case $ U = \Omega $.
For the same reason, we know that
\begin{equation*}
\duality{\norm{T u_1}{}{TH^1(\Omega)} g_2}{\overline{T u_2}} - \overline{\duality{N_{\overline{A_2}, \overline{q_2}} u_2}{\overline{\norm{T u_1}{}{TH^1(\Omega)} f_2}}} = 0
\end{equation*}
holds for every $ (f_2, g_2) \in C_{A_2, q_2} $. Last identity immediately imply
\begin{gather*}
\left| \duality{N_{A_1, q_1} u_1}{\overline{T u_2}} - \overline{\duality{N_{\overline{A_2}, \overline{q_2}} u_2}{\overline{T u_1}}} \right| \\
\leq \norm{N_{A_1, q_1} u_1 - \norm{T u_1}{}{TH^1(\Omega)} g_2}{}{TH^1(\Omega)^\ast} \norm{T u_2}{}{TH^1(\Omega)} \\
+ \norm{N_{\overline{A_2}, \overline{q_2}} u_2}{}{TH^1(\Omega)^\ast} \norm{T u_1 - \norm{T u_1}{}{TH^1(\Omega)} f_2}{}{TH^1(\Omega)}.
\end{gather*}
On the other hand
\begin{equation*}
\norm{N_{\overline{A_2}, \overline{q_2}} u_2}{}{TH^1(\Omega)^\ast} \lesssim \left[ 1 + \norm{A_2}{2}{L^\infty(\Omega; \C^n)} + \norm{q_2}{}{L^\infty(\Omega)} \right] \norm{u_2}{}{H^1(\Omega)}.
\end{equation*}
Since
\[ \norm{T u_j}{}{TH^1(\Omega)} \leq \norm{u_j}{}{H^1(\Omega)} \]
for $ j \in \{ 1, 2 \} $, we have that
\begin{gather*}
\left| \duality{N_{A_1, q_1} u_1}{\overline{T u_2}} - \overline{\duality{N_{\overline{A_2}, \overline{q_2}} u_2}{\overline{T u_1}}} \right| \lesssim I \left((f_1, g_1) ; C_{A_2, q_2} \right) \\
\times \left[ 1 + \norm{A_2}{2}{L^\infty(\Omega; \C^n)} + \norm{q_2}{}{L^\infty(\Omega)} \right] \norm{u_1}{}{H^1(\Omega)} \norm{u_2}{}{H^1(\Omega)}.
\end{gather*}
where
\[ f_1 = \frac{Tu_1}{\norm{T u_1}{}{TH^1(\Omega)}}, \qquad g_1 = \frac{N_{A_1, q_1} u_1}{\norm{T u_1}{}{TH^1(\Omega)}}. \]
Now the statement of the proposition follows easily using \eqref{id:startBOUNDARY} and taking supremum and then maximum.
\end{proof}

\section{Complex geometric optics solutions}\label{sec:CGOs}
In this section, we review the properties of the CGOs constructed by Krupchyk
and Uhlmann in \cite{KU} for the particular case where the magnetic potential
satisfies a prescribed $ L^2 $-modulus of continuity. 
The additional regularity allows us to attain appropriate remainder estimates
that are needed later. We end the section by estimating the $ H^1 $-norm of these CGOs.

Throughout this section we assume that $ q \in L^\infty(U) $, $ A \in L^\infty(\R^n; \C^n) $ and
\[ \supp A \subset \overline{U}, \]
where $ U \subset \R^n $ is a domain.
For notational convenience, we write throughout this section $ \norm{q}{}{L^\infty} $ and $ \norm{A}{}{L^\infty} $ to denote the norms of $ q \in L^\infty(U) $ and $ A \in L^\infty(\R^n; \C^n) $, respectively. In addition, we assume that $ |A|^2_{B^{2, \infty}_\varepsilon} < \infty $. The definition of $|\cdot|_{B^{2, \infty}_\varepsilon}$ was given in Section \ref{sec:intro}.

Let $ \Psi $ belong to $ C^\infty_0(\R^n) $ with $ 0 \leq \Psi(x) \leq 1 $ for
all $ x \in \R^n $, $ \supp \Psi \subset \{ x \in \R^n : |x| \leq 1 \} $ and $
\int_{\R^n} \Psi \, dx = 1 $. Define $\Psi_\tau (x) = \tau^{-n} \Psi(x / \tau)
$ for $\tau \in (0, 1]$ and $ x \in \R^n $. Then $ A^\sharp = \Psi_\tau \ast A
\in C_0^\infty(\R^n,\C^n) $ (where convolution is taken with each component of
$ A $) and $ A^\flat = A - A^\sharp $ satisfies
\begin{equation} \label{es:Aflat}
\norm{A^\flat}{}{L^2(\R^n; \C^n)} \leq \tau^\varepsilon |A|_{B^{2, \infty}_\varepsilon}
\end{equation}
for $ \tau \in (0, 1] $. On the other hand,
\begin{equation} \label{es:part-deriv_Asharp}
\norm{\partial^\alpha A^\sharp}{}{L^\infty(\R^n; \C^n)} \lesssim \tau^{-|\alpha|} \norm{A}{}{L^\infty}
\end{equation}
for $ \tau \in (0, 1] $ and $ \alpha \in \N^n $, where the implicit constant in
this inequality only depends on $ \Psi $.

In \cite{KU}, Krupchyk and Uhlmann proved the existence of CGOs in $ H^1(U) $ solving
\[ L_{A, q} u = 0 \]
with $ A \in L^\infty (U; \C^n) $ and $ q \in L^\infty (U) $. These CGOs are solutions of the form
\begin{equation} \label{eq:CGOform}
u(x;\zeta,h) = e^{x\cdot\zeta/h} (a(x;\zeta,h) + r(x;\zeta,h))
\end{equation}
where $\zeta\in\C^n$ with $\zeta\cdot\zeta=0$ and $|\zeta|\sim 1$; $ h $ is a
small positive parameter; $a$ is a smooth amplitude and $r$ is a correction
term. In the next lines we follow Krupchyk and Uhlmann's ideas to check the
properties of $ u(\cdot; \zeta, h) $ in the particular case where $ A $
and $ q $ are as the beginning of the section.

Let the restriction of $ A $ to $ U $ be also denoted by $ A $. Consider $
\zeta = \zeta_0 + \zeta_1 $ with $ \zeta_0 $ independent of $ h $, $
\mathrm{Re}\, \zeta_0 \cdot \mathrm{Im}\, \zeta_0 = 0 $, $ |\mathrm{Re}\,
\zeta_0| = |\mathrm{Im}\, \zeta_0| = 1 $ and $ |\zeta_1| = \mathcal{O} (h) $ as
$ h $ becomes small. In order to construct $ u(\cdot; \zeta, h) $ of the
form of \eqref{eq:CGOform} satisfying $ L_{A, q} u = 0 $, it is enough to prove
the existence of a
$ r(\cdot; \zeta, h, \tau) \in H^1(U) $ solving
\begin{equation}\label{eq:remainder}
e^{- \zeta \cdot (\centerdot) / h} h^2 L_{A, q} ( e^{\zeta \cdot (\centerdot) / h} r) =
e^{- \zeta \cdot (\centerdot) / h} h^2 L_{A, q} ( e^{\zeta \cdot (\centerdot) / h} a) =: w,
\end{equation}
in $\R^n$.
One does this by first finding an 
$ a(\cdot; \zeta_0, \tau) \in C^\infty(\R^n) $ that solves
\begin{equation}\label{eq:transport}
\zeta_0 \cdot \nabla a + i \zeta_0 \cdot A^\sharp a = 0,
\end{equation}
so that $w$ becomes
\begin{align*}
w = & h^2\Delta a + i h^2 A \cdot \nabla a - h^2 m_A (a) - h^2 (A^2+q) a + 2h \zeta_1 \cdot \nabla a \\
 & + 2 h i\zeta_0 \cdot A^\flat a + 2 hi \zeta_1 \cdot A a,
\end{align*}
where $ m_A $ denotes the bounded linear operator from $ H^1(U) $ to $
H^{-1}(U) $ defined by
\[ \duality{m_A (\phi)}{\psi} = \int_{U} i \phi A \cdot \nabla \psi \, dx \]
for all $ \phi \in H^1(U) $ and all $ \psi \in H^1_0(U) $.

If we look for solutions to \eqref{eq:transport} in the form
\[ a(\cdot; \zeta_0, \tau) = e^{\Phi^\sharp(\cdot; \zeta_0, \tau)}, \]
it will be enough that $ \Phi^\sharp(\cdot; \zeta_0, \tau) $ satisfies
\begin{equation}
\zeta_0 \cdot \nabla \Phi^\sharp + i \zeta_0 \cdot A^\sharp = 0 \label{eq:deltabarSHARP}
\end{equation}
in $ \R^n $. Since $ \mathrm{Re}\, \zeta_0 \cdot \mathrm{Im}\, \zeta_0 = 0 $
and $ |\mathrm{Re}\, \zeta_0| = |\mathrm{Im}\, \zeta_0| = 1 $, $ \zeta_0 \cdot
\nabla $ is a $ \overline{\partial} $ operator in suitable coordinates.
Therefore, $ \Phi^\sharp = (\zeta_0 \cdot \nabla)^{-1} (- i \zeta_0 \cdot
A^\sharp) $ belongs to $ C^\infty (\R^n) $ and, using
\eqref{es:part-deriv_Asharp}, we have that
\begin{equation} \label{es:phi_sharp}
\norm{\partial^\alpha \Phi^\sharp}{}{L^\infty(\R^n)} \lesssim \tau^{- |\alpha|} \norm{A}{}{L^\infty}
\end{equation}
for $ \tau \in (0, 1] $ and $ \alpha \in \N^n $ (For more details see Lemma 4.6
in \cite{Sa} and Lemma 2.1 in \cite{Su}). Here the implicit constant only
depends on $ \alpha $. Moreover, $ \Phi (\cdot; \zeta_0) = (\zeta_0 \cdot
\nabla)^{-1} (- i \zeta_0 \cdot A) \in L^\infty (\R^n) $ solves
\begin{equation}
\zeta_0 \cdot \nabla \Phi + i \zeta_0 \cdot A = 0 \label{eq:deltabar}
\end{equation}
and  satifies
\begin{align}
& \norm{\Phi (\cdot; \zeta_0)}{}{L^\infty(\R^n)} \lesssim \norm{A}{}{L^\infty}, \label{es:phi} \\
& \norm{\chi (\Phi^\sharp (\cdot, \zeta_0, \tau) - \Phi (\cdot; \zeta_0))}{}{L^2(\R^n)} \lesssim \tau^\varepsilon |A|_{B^{2, \infty}_\varepsilon}, \label{es:PhiPhiSharp}
\end{align}
for any $ \chi \in C^\infty_0(\R^n) $. The implicit constant in \eqref{es:PhiPhiSharp} depends on $ \chi $ and $ U $. The estimate \eqref{es:PhiPhiSharp} is an immediate consequence of Lemma 3.1 in \cite{SyU} and the estimate \eqref{es:Aflat}.

Regarding equation \eqref{eq:remainder}, Krupchyk and Uhlmann proved (see
Proposition 2.3 in \cite{KU}) that there exists a positive decreasing function $
h_0 $ defined in $ (0, +\infty) \subset \R $ such that, for all $ h \leq h_0 (
\norm{A}{}{L^\infty(U; \C^n)} ) $, there exists $ r(\cdot; \zeta, h, \tau) $
which is a $ H^1(U) $ solution to \eqref{eq:remainder} and satisfies
\begin{equation} \label{es:remainder}
\norm{r}{}{H^1_\mathrm{scl} (U)} \lesssim \frac{1}{h} \norm{w}{}{H^{-1}_\mathrm{scl}(U)}.
\end{equation}
Here the implicit constant depends on $ U $. The semi-classical norms are defined by
\begin{align*}
\norm{r}{2}{H^1_\mathrm{scl} (U)} &= \norm{r}{2}{L^2(U)} + \norm{h \nabla r}{2}{L^2(U; \C^n)}, \\
\norm{w}{}{H^{-1}_\mathrm{scl}(U)} &= \sup_{\phi \in H^1_0(U) \setminus \{ 0
\}} \frac{| \langle w , \phi \rangle |}{\norm{\phi}{}{H^1_\mathrm{scl}(U)}}.
\end{align*}

On of the key properties of a CGO solution is that the correction term $r$
tends to vanish, in some sense, when the parameter $ h $ becomes small. This
can be deduced from \eqref{es:remainder} by computing 
$\norm{w}{}{H^{-1}_\mathrm{scl}(U)} $ and choosing $ \tau $ as a proper power of
$ h $:\\
By the Cauchy-Schwarz inequality and estimate \eqref{es:phi_sharp} we can prove
that there exists a constant $c > 0$ such that, for any $ \phi \in H^1_0(U)$,
\begin{gather*}
\left| \duality{h^2\Delta a + i h^2 A \cdot \nabla a + 2h \zeta_1 \cdot \nabla
a + 2 hi \zeta_1 \cdot A a}{\phi} \right| \\
\lesssim \frac{h^2}{\tau^2} e^{c\norm{A}{}{L^\infty}} \left( 1 +
\norm{A}{2}{L^\infty} \right) \norm{\phi}{}{H^1_\mathrm{scl}(U)}.
\end{gather*}
Again by the Cauchy-Schwarz inequality and \eqref{es:phi_sharp}, there exists a
constant $ c > 0 $ such that, for any $ u \in H^1_0(U) $,
\begin{equation*}
\left| \duality{h^2 (A^2+q) a}{\phi} \right| \lesssim h^2 \left( \norm{A}{2}{L^\infty} + \norm{q}{}{L^\infty} \right) e^{c\norm{A}{}{L^\infty}} \norm{\phi}{}{H^1_\mathrm{scl}(U)}.
\end{equation*}
By the Cauchy-Schwarz inequality and the estimates \eqref{es:Aflat} and \eqref{es:phi_sharp}, there exists $ c > 0 $ such that, for any $ \phi \in H^1_0(U) $,
\[ \left| \duality{2 h i\zeta_0 \cdot A^\flat a}{\phi} \right| \lesssim h \tau^\varepsilon e^{c\norm{A}{}{L^\infty}} |A|_{B^{2, \infty}_\varepsilon} \norm{\phi}{}{H^1_\mathrm{scl}(U)}. \]
Finally, by integrating by parts, Cauchy-Schwarz inequality, estimate
\eqref{es:phi_sharp}, \eqref{es:part-deriv_Asharp} and \eqref{es:Aflat}, there
exists a $ c > 0 $ such that, for any $ \phi \in H^1_0(U) $,
\begin{align*}
\left| \duality{h^2 m_A (a)}{\phi} \right| &\leq \left| \int_{U} i h^2 \nabla \cdot ( a A^\sharp ) \phi \, dx \right| + \left| \int_{U} i h^2 a A^\flat \cdot \nabla \phi \, dx \right| \\
&\lesssim e^{c\norm{A}{}{L^\infty}} \left[ \frac{h^2}{\tau} + \frac{h^2 \norm{A}{2}{L^\infty}}{\tau} + h \tau^\varepsilon |A|_{B^{2, \infty}_\varepsilon} \right] \norm{\phi}{}{H^1_\mathrm{scl}(U)}.
\end{align*}
The implicit constant in the last four inequalities depends on $ U $.
Therefore, choosing $ \tau = h^{1/(\varepsilon + 2)} $ in the above estimates
and using \eqref{es:remainder}, we see that, for $ h \leq h_0 (
\norm{A}{}{L^\infty (U; \C^n)} ) $,
\begin{equation}\label{es:remainder2}
\norm{r}{}{H^{1}_\mathrm{scl}(U)} \lesssim h^\frac{\varepsilon}{\varepsilon + 2} e^{c\norm{A}{}{L^\infty}} \left[1 + \norm{A}{2}{L^\infty} + \norm{q}{}{L^\infty} + |A|_{B^{2, \infty}_\varepsilon} \right].
\end{equation}

We end this section by estimating the $ H^1(U) $-norm of $ u(\cdot; \zeta, h) $:
\begin{equation}\label{es:uH1}
\norm{u}{}{H^{1}(U)} \lesssim e^{c'/h} e^{c\norm{A}{}{L^\infty}} 
\left[1 + \norm{A}{2}{L^\infty} + \norm{q}{}{L^\infty} + |A|_{B^{2, \infty}_\varepsilon} \right]
\end{equation}
where $ c' > 0 $ and the implicit constant depend on $ U $.

\section{Stability estimates for the magnetic fields}\label{sec:stabilityMAG}

The aim of this section is to prove Theorem \ref{th:magnetic} by deriving the two 
stability estimates for the magnetic fields.
The first step will be to use Proposition \ref{prop:integralID} and the CGOs
constructed in Section \ref{sec:CGOs} to estimate the Fourier transform of the
difference of the magnetic fields. Then, we prove the stability estimates in
Sobolev (the general approach follows \cite{Al}) and Besov spaces.

Consider an a priori constant $ M \in  [1, + \infty) $ and a small constant $
\varepsilon \in (0, 1) $. Let $ A_1, A_2 \in L^\infty(\Omega; \C^n) $ be
two magnetic potentials and let $ q_1, q_2 \in L^\infty(\Omega) $ be two
electric potentials. Assume that the extension by zero of $ A_j $ out of $
\Omega $, still denoted by $A_j$, satisfies $|A_j|_{B^{2, r}_\varepsilon} <
\infty $ with $r \in [1, +\infty)$ or $r = \infty$. Furthermore, assume that
\begin{equation}
\norm{A_j}{}{L^\infty (\R^n; \C^n)} + |A_j|_{B^{2, r}_\varepsilon} +
\norm{q_j}{}{L^\infty(\Omega)} \leq M \label{es:A_japriori}
\end{equation}
for $ j \in \{ 1, 2 \} $. 
The implicit constants in the inequalities may, throughout this section, depend 
on $ M $ and $ \varepsilon $, as well as on $ n $ and $ \Omega $.

For any $\xi\in\R^n$, consider $ \mu_1,\mu_2 \in \R^n $ such that
$|\mu_1|=|\mu_2|=1$ and $\mu_1\cdot\mu_2=\mu_1\cdot\xi=\mu_2\cdot\xi=0$. 
For any positive $ h $ with $ h \leq \min(1, 2/|\xi|) $, we define
\begin{equation}\label{eq:zeta_1_2}
\begin{aligned}
&\zeta_1=\frac{ih\xi}{2}+\mu_1 + i\sqrt{1-h^2\frac{|\xi|^2}{4}}\mu_2 \\
&\zeta_2=-\frac{ih\xi}{2}-\mu_1+i\sqrt{1-h^2\frac{|\xi|^2}{4}}\mu_2.
\end{aligned}
\end{equation}
Note that $\zeta_j\cdot\zeta_j=0$ for $j \in \{1, 2 \}$ and
$(\zeta_1+\overline{\zeta_2})/h=i\xi$. Moreover, $\zeta_1= \mu_1+
i\mu_2+\mathcal{O}(h)$ and $\zeta_2= -\mu_1+ i\mu_2+\mathcal{O}(h)$.

Let
\begin{equation}\label{for:CGO1}
u_1(x;\zeta_1,h) = e^{x\cdot\zeta_1/h} (e^{\Phi^\sharp_1(x; \mu_1+i\mu_2, h)} + r_1(x;\zeta_1,h))
\end{equation}
and
\begin{equation}\label{for:CGO2}
u_2(x;\zeta_2,h) = e^{x\cdot\zeta_2/h} (e^{\Phi^\sharp_2(x; -\mu_1+i\mu_2, h)} + r_2(x;\zeta_2,h))
\end{equation}
be CGO solutions of $ L_{A_1, q_1} u_1 = 0 $ and $ L_{\ov{A_2}, \ov{q_2}} u_2 = 0 $ --constructed\footnote{Note that the bounded inclusion $ B^{2, r}_\varepsilon (\R^n) \hookrightarrow B^{2, \infty}_\varepsilon (\R^n) $ and \eqref{es:A_japriori} provide an a priori bound for $ |A_j|_{B^{2, \infty}_\varepsilon} $ depending on $ M $.} as in Section \ref{sec:CGOs} for $ U = \Omega $ and $ \tau = h^{1/(\varepsilon + 2)} $.

We now state the estimate for the Fourier transform of the difference of the
magnetic fields.
Notice that we use the notations $A_1$ and $A_2$ for both the vector fields and the corresponding 
1-forms, depending on the context.
\begin{lemma}\label{dAFourier} \sl
There exists a constant $ c > 0 $ depending on $ \Omega $ such that
\[
\big| \widehat{dA_1}(\xi) - \widehat{dA_2}(\xi) \big|  \lesssim |\xi| \left( \dist (C_1,C_2) e^{c/h} 
    + h^{\varepsilon/(\varepsilon + 2)} \right)
\]
for all $ h \leq \min(1, 2/|\xi|, h_0(M)) $.
\end{lemma}
\begin{proof} To prove the statement we just need to plug in $ u_1 $ and $ u_2 $,
as in \eqref{for:CGO1} and \eqref{for:CGO2}, in the estimate of Proposition
\ref{prop:integralID} multiplied by $ h $ and then study the behaviour in $
h $. The term $u_1 \ov{u_2}$ is bounded in $h$, since $(\zeta_1+\overline{\zeta_2})/h=i\xi$. 
One sees then by the 
Cauchy-Schwarz inequality, \eqref{es:phi_sharp} and \eqref{es:remainder2},
that
\begin{equation}
\left| \int_{\Omega} (A_1^2 - A_2^2 + q_1 - q_2) u_1 \overline{u_2}  \, dx \right| \lesssim 1.
\label{es:acotada1}
\end{equation}
Therefore, by Proposition \ref{prop:integralID}, \eqref{es:uH1} and
\eqref{es:acotada1}, there exists a constant $ c > 0 $, that depends on $
\Omega $, such that
\begin{equation*}
h \left| \int_{\Omega} (A_1 - A_2) \cdot (u_1 \nabla \overline{u_2} -
\overline{u_2} \nabla u_1) \, dx \right| \lesssim \dist(C_1, C_2) e^{c/h} + h.
\end{equation*}
Again by Cauchy-Schwarz inequality, \eqref{es:phi_sharp} and
\eqref{es:remainder2}; one can estimate the left hand side of last
estimate from bellow as follows
\begin{align*}
\bigg| \int_{\Omega} &(A_1-A_2)\cdot(\ov{\zeta_2}-\zeta_1) e^{ix\cdot\xi}
    e^{\Phi_{1}^\sharp+\ov{\Phi_{2}^\sharp}}dx  \bigg| \\
&\lesssim h \bigg| \int_{\Omega} (A_1 - A_2) \cdot (u_1 \nabla \overline{u_2} -
    \overline{u_2} \nabla u_1) \, dx \bigg| + h^{\varepsilon / (\varepsilon +
    2)}.
\end{align*}
Now we want to replace $\Phi^\sharp_j$ by $\Phi_j$ on the right hand side of this estimate. Since $\ov{\zeta_2}-\zeta_1 = -2(\mu_1 + i \mu_2) + \mathcal{O}(h)$ we have, by \eqref{es:phi_sharp}, that
\begin{align*}
\bigg|
\int_{\Omega}& (A_1-A_2)\cdot (\mu_1 + i \mu_2) 
e^{ix\cdot\xi} e^{\Phi_{1}+\ov{\Phi_{2}}}dx 
\bigg| \\
& \lesssim  
\bigg|
\int_{\Omega} (A_1-A_2)\cdot(\ov{\zeta_2}-\zeta_1) 
e^{ix\cdot\xi} e^{\Phi_{1}^\sharp+\ov{\Phi_{2}^\sharp}}dx \bigg|\\
&\quad+ 
\bigg|
\int_{\Omega} (A_1-A_2)\cdot (\mu_1 + i \mu_2) 
e^{ix\cdot\xi} \big( e^{\Phi_{1}^\sharp+\ov{\Phi_{2}^\sharp}} - e^{\Phi_{1}+\ov{\Phi_{2}}} \big) dx\bigg| + h.
\end{align*}
The second integral on the right hand side can be 
estimated as
\begin{equation*}
\bigg|
\int_{\Omega} (A_1-A_2)\cdot (\mu_1 + i \mu_2) 
e^{ix\cdot\xi} \big( e^{\Phi_{1}^\sharp+\ov{\Phi_{2}^\sharp}} - e^{\Phi_{1}+\ov{\Phi_{2}}} \big) dx \bigg| \lesssim  h^{\varepsilon/(\varepsilon+2)}.
\end{equation*}
using \eqref{es:PhiPhiSharp}, \eqref{es:phi_sharp}, \eqref{es:phi} and the inequality
\begin{equation}
|e^{z_1} - e^{z_2}| \leq |z_1 - z_2| e^{\max (\mathrm{Re} z_1, \mathrm{Re} z_2)}.
\label{in:exponential}
\end{equation}
Thus, we may write
\begin{equation*}
\bigg|
(\mu_1 + i \mu_2)  \cdot \int_{\Omega} (A_1-A_2)
e^{ix\cdot\xi} e^{\Phi_{1}+\ov{\Phi_{2}}}dx 
\bigg|
\lesssim \dist(C_1, C_2) e^{c/h} + h^{\varepsilon/(\varepsilon+2)}.
\end{equation*}
Now by Proposition 3.3 in \cite{KU} we can remove $ e^{\Phi_{1}+\ov{\Phi_{2}}} $ and get that
\begin{equation*}
\bigg|
(\mu_1 + i \mu_2) \cdot \int_{\Omega} (A_1-A_2)
e^{ix\cdot\xi} dx
\bigg|
\lesssim \dist(C_1, C_2) e^{c/h} + h^{\varepsilon/(\varepsilon+2)}.
\end{equation*}
To finish the proof, note that the above computations also hold if we replace $\mu_1 + i \mu_2$ by $\mu_1 - i \mu_2$, hence
\begin{equation*}
\bigg|
\mu \cdot \int_{\Omega} (A_1-A_2)
e^{ix\cdot\xi} dx
\bigg|
\lesssim \dist(C_1, C_2) e^{c/h} + h^{\varepsilon/(\varepsilon+2)}
\end{equation*}
for any unit vector $\mu$ such that $\mu \cdot \xi =0$.
In particular, it holds for the vectors $\mu_{j,k} = (\xi_j^2 + \xi_k^2)^{-1/2} (\xi_j e_k - \xi_k e_j)$ with $ j, k \in \{ 1, \dots, n \} $, since $\mu_{j,k} \cdot \xi = 0$. Here $\xi_j$ denotes the $j$-th component of $\xi$ and $e_k$ the $k$-th element of the canonical basis of $\R^n$. Thus
\begin{equation*}
\left| \widehat{dA_1} (\xi) - \widehat{dA_2} (\xi) \right| \lesssim |\xi| \left( \dist(C_1, C_2) e^{c/h} + h^{\varepsilon/(\varepsilon+2)} \right).
\end{equation*}
\end{proof}

Next we derive the stability estimate for the difference of the magnetic fields
in the Sobolev space $ H^{-1}\boldsymbol\Omega^2(\R^n) $ using the equivalent
norm given by
\[
    \|u\|^2_{H^{-1}\boldsymbol\Omega^2(\R^n)} = \int_{\R^n} (1+|\xi|^2)^{-1} |\widehat u(\xi)|^2 \, d\xi.
\]

\begin{proposition}\label{dAStab} \sl
There exist constants $ c > 1 $ depending on $ \Omega $ and $ 0< \tilde{c} < 1 $ universal such that 
\[
    \|  d A_1- d A_2  \|_{H^{-1}\boldsymbol\Omega^2(\R^n)}  
    \lesssim 
        \big|  \log \dist(C_1,C_2) \big|^{-\tilde{c}\varepsilon^2 / n},
\]
provided that
\[ \big|  \log \dist(C_1,C_2) \big|^{-1} \leq 1 / c \min(1, h_0(M)) .\]
\end{proposition}
\begin{proof}
Let $B_\rho$ be denote the ball centred at $ 0 \in \R^n $ of radius $ \rho \geq 1 $ and let $B_\rho^c$ denote its complement in $ \R^n $. Let $A$ denote $A_1-A_2$ for clarity.
Using Lemma \ref{dAFourier} we may estimate
\begin{align}\label{es:Brho}
    \int_{B_\rho} 
    \frac{\big|\widehat{dA}(\xi)\big|^2}{1+|\xi|^2} d\xi
    \lesssim \rho^{n} \big( \dist(C_1,C_2) e^{c/h} + h^{\varepsilon/(\varepsilon + 2)} \big)^2
\end{align}
for all $ h \leq \min(1, 2/\rho, h_0(M)) $. Note that this $ c $ does not denote the one in the statement. On the other hand, write $ A = A^\sharp + A^\flat $ using the same notation as in Section \ref{sec:CGOs}, where the parameter $ \tau $ here is to be chosen. Then
\begin{align*}
    \int_{B_\rho^c}
    \frac{\big|\widehat{dA}(\xi) \big|^2}{1+|\xi|^2} d\xi 
    & \lesssim
    \int_{B_\rho^c}  
    \frac{|\widehat{dA^\sharp}|^2}{1+|\xi|^2} d\xi 
    + 
    \int_{B_\rho^c}  
    \frac{|\xi|^2 |\widehat{A^\flat}|^2}{1+|\xi|^2} d\xi \\
    & \lesssim
    \rho^{-2} \norm{dA^\sharp}{2}{L^2\boldsymbol\Omega^2(\R^n)} + \norm{A^\flat}{2}{L^2\boldsymbol\Omega^1(\R^n)}
    \end{align*}
Since the $ \supp A $ is compact, estimates \eqref{es:part-deriv_Asharp} and \eqref{es:Aflat} imply that
\begin{equation}\label{es:Bcrho}
	\int_{B_\rho^c}
    \frac{\big|\widehat{dA}(\xi) \big|^2}{1+|\xi|^2} d\xi
    \lesssim
    \rho^{-2} \tau^{-2} + \tau^{2\varepsilon}.
\end{equation}
Choosing $\tau = \rho^{-1/(\varepsilon+1)}$, we have
\begin{align*}
    \|  dA  \|^2_{H^{-1}\boldsymbol\Omega^2(\R^n)} 
    \lesssim \rho^{n} \dist(C_1,C_2)^2 e^{2c/h} + \rho^{n} h^{2\varepsilon/(\varepsilon + 2)}
    + \rho^{-2\varepsilon/(\varepsilon+1)}
\end{align*}
by \eqref{es:Brho} and \eqref{es:Bcrho}.
By equating the two last terms on the right hand side we express $\rho$ in terms of $h$ as
\[ \rho = h^{- \frac{2 \varepsilon (1+ \varepsilon)}{(2 + \varepsilon)(n + n\varepsilon + 2\varepsilon)}}, \]
which gives
\begin{align*}
    \|  dA  \|^2_{H^{-1}\boldsymbol\Omega^2(\R^n)} 
    \lesssim  \dist(C_1,C_2)^2 e^{c'/h} 
    + h^{4\varepsilon^2 / ((\varepsilon+2)(n + n\varepsilon + 2\varepsilon))}
\end{align*}
for $ c' > 2c $. Note that this choice of $ \rho $ satisfies the restriction $ \rho \leq 2/h $.
Finally, we set
\[ h = c' |\log \dist(C_1, C_2)|^{-1}\]
to prove the statement.
\end{proof}

We next derive the stability estimate for the difference of the magnetic fields
in the Besov space $ B^{2, r}_{-\delta}\boldsymbol\Omega^2 (\R^n) $ with $ 1 >
\delta > (1- \varepsilon) $ and the norm given by
\[ \norm{u}{r}{B^{2, r}_{-\delta}\boldsymbol\Omega^2 (\R^n)} = \sum_{j \in \N}
2^{-r\delta j} \norm{\Delta_j u}{r}{L^2\boldsymbol\Omega^2 (\R^n)} \]
for $ r \in [1, +\infty) $ and
\[ \norm{u}{}{B^{2, \infty}_{-\delta}\boldsymbol\Omega^2 (\R^n)} = \sup_{j \in
\N} \left( 2^{-\delta j} \norm{\Delta_j u}{}{L^2\boldsymbol\Omega^2 (\R^n)}
\right) \]
for $ r = \infty $.

In the following lines we describe the family of operators
$ \{ \Delta_j \}_{j \in \N} $. We begin by 
picking  a smooth  cut-off function $\eta$ defined in $
\R^n $ such that $ \eta (\xi) = 1 $ for $ |\xi| \leq 1 $ and $ \eta(\xi) = 0 $
for $ |\xi| \geq 2  $ and $ \kappa $ being defined as $ \kappa (\xi) =
\eta(\xi) - \eta(2 \xi) $. Note that $ \kappa $ is supported in the shell $ \{
\xi \in \R^n : 1/2 \leq |\xi| \leq 2 \} $ and $ \kappa (2^{-j} \cdot) $ is
supported in $ \{ \xi \in \R^n : 2^{j - 1} \leq |\xi| \leq 2^{j + 1} \} $. 
Notice that it follows from the definitions that these
functions form a partition of unity, i.e.
\[ 1 = \eta (\xi) + \sum_{j \in \N \setminus \{ 0 \}} \kappa(2^{-j} \xi) \]
for all $ \xi \in \R^n $. Finally let
$ \psi_0 $ be defined as $ \widehat{\psi_0}(\xi) = \eta (\xi) $ and let $
\psi_j $ with $ j \in \N \setminus \{ 0 \} $ be defined as $
\widehat{\psi_j}(\xi) = \kappa (2^{-j} \xi) $. The operator $ \Delta_j $ with $
j \in \N $ is then defined as $ \Delta_j u = \psi_j \ast u $.
\begin{proposition} \label{prop:dAbesov} \sl
There exist constants $ c > 1 $ depending on $ \Omega $  and $ 0< \tilde{c} < 1 $ universal such that 
\[
    \|  d A_1- d A_2  \|_{B^{2, r}_{-\delta} \boldsymbol\Omega^2(\R^n)}  
    \lesssim 
	\big|  \log \dist(C_1,C_2) \big|^{-\tilde{c}\varepsilon (\delta - 1 + \varepsilon) / n},
\]
provided that
\[ \big|  \log \dist(C_1,C_2) \big|^{-1} \leq 1 / c \min(1, h_0(M)) .\]
The implicit constant above depends also on $ \delta $.
\end{proposition}

\begin{proof} Let $A$ denote $A_1-A_2$ for clarity. Consider $ k \in \N $ to be chosen later. For any $ j \in \N $ such that $ j \leq k $ we have by Lemma \ref{dAFourier} that
\[ 2^{-\delta j} \norm{\Delta_j dA}{}{L^2\boldsymbol\Omega^2 (\R^n)} \lesssim 2^{k (n/2 + 1)} \big( \dist(C_1,C_2) e^{c/h} + h^{\varepsilon/(\varepsilon + 2)} \big) \]
for all $ h \leq \min(2^{-k}, h_0(M)) $. Note that this $ c $ does not denote the one in the statement. On the other hand, if $ j > k $, then
\[ 2^{-\delta j} \norm{\Delta_j dA}{}{L^2\boldsymbol\Omega^2 (\R^n)} \lesssim 2^{-k(\delta - 1 + \varepsilon)} 2^{j \varepsilon} \norm{\Delta_j A}{}{L^2\boldsymbol\Omega^1 (\R^n)} \]
since $\delta - 1 + \varepsilon > 0$. Thus,
\[ \norm{dA}{}{B^{2, r}_{-\delta}\boldsymbol\Omega^2 (\R^n)} \lesssim 2^{k n} \dist(C_1,C_2) e^{c/h} + 2^{k n} h^{\varepsilon/(\varepsilon + 2)} + 2^{-k(\delta - 1 + \varepsilon)} \]
for all $ h \leq \min(2^{-k}, h_0(M)) $. Now choosing $ k \in \N $ such that
\[ 2^{-(k+1)} < h^\frac{\varepsilon}{(\varepsilon + 2)(\delta - 1 + \varepsilon + n)} \leq 2^{-k} \]
we know that there exists a constant $ c' > 0 $ such that
\[ \norm{dA}{}{B^{2, \infty}_{-\delta}\boldsymbol\Omega^2 (\R^n)} \lesssim \dist(C_1,C_2) e^{c'/h} + h^\frac{\varepsilon (\delta - 1 + \varepsilon)}{(\varepsilon + 2)(\delta - 1 + \varepsilon + n)}. \]
Note that the choice of $ k $ satisfies the restriction $ h \leq 2^{-k} $. Finally, we set
\[ h= 2c'|\log \dist(C_1, C_2)|^{-1} \]
to prove the statement.
\end{proof}

\section{Stability estimates for the electric potentials}\label{sec:stabilityELEC}
In this section we prove Theorem \ref{th:electric} by deriving the two stability
estimates for the electric potentials. 
Our starting point could again be the estimate given in Proposition
\ref{prop:integralID}. There are however some difficulties with this.
It seems that in order to isolate in that inequality the
difference $q_1-q_2$ we would need to control the difference
$A_1-A_2$. Unfortunately we can only control the difference
of the magnetic fields $dA_1-dA_2$. To overcome this difficulty we give a slight
modification of the estimate in Proposition \ref{prop:integralID}. This
modification is based on the invariance of the Cauchy data sets under gauge
transformations in an open ball $ B $ containing $ \ov{\Omega} $ (see
Lemma \ref{lem:gaugeINVA} below).
Then, we use the CGOs
constructed in Section \ref{sec:CGOs} to estimate the Fourier transform of the
difference of the electric potentials. Finally, we prove the stability
estimates in Sobolev and Besov spaces.

As in the previous section we consider an a priori constant 
$ M \in  [1, + \infty) $ and a small constant 
$\varepsilon \in (0, 1) $. Let $ A_1, A_2 \in L^\infty(\Omega; \C^n) $ be
two magnetic potentials and let $ q_1, q_2 \in L^\infty(\Omega) $ be two
electric potentials. Assume that the extension by zero of $ A_j $ out of $
\Omega $, still denoted by $A_j$, satisfies $|A_j|_{B^{2, r}_\varepsilon} <
\infty $ with $r \in [1, +\infty)$ or $r = \infty$. Let $ q_1 $ and $ q_2 $
also denote the extensions by zero of the electric potentials.
Furthermore, assume that
\begin{equation}
\norm{A_j}{}{L^\infty (\R^n; \C^n)} + |A_j|_{B^{2, r}_\varepsilon} + \norm{q_j}{}{L^\infty(\R^n)} \leq M
\label{es:A_japriori2}
\end{equation}
for $ j \in \{ 1, 2 \} $. Throughout this section, the constants implicit in each inequality may depend on $ M $ and $ \varepsilon $, as well as on $ n, \Omega $ and on an open ball $ B $ containing $ \ov{\Omega} $.

For notational convenience, the norms $ \norm{\cdot}{}{L^p(B)} $ and $ \norm{\cdot}{}{L^p(B;\C^n)} $ with $ 1 \leq p \leq \infty $ will be denoted by $ \norm{\cdot}{}{L^p} $ and we will write $ (A_2 + \nabla\varphi)^2 = (A_2 + \nabla\varphi) \cdot (A_2 + \nabla\varphi) $.
\begin{lemma} \label{lem:gaugeINVA} \sl
Let $ B $ denote an open ball containing $ \ov{\Omega} $ and let $ \varphi $
belong to $ W^{1,n}(B) \cap L^\infty (B) $ with $\varphi|_{\partial B} = 0$.
Then, for any $ u_1, u_2 \in H^1(B) $ solving $ L^B_{A_1, q_1} u_1 = 0 $ and $
L^B_{\overline{A_2}, \overline{q_2}} u_2 = 0 $, we have
\begin{align*}
    \bigg| \int_{B} &i e^{i \varphi} \big(A_1 - (A_2 + \nabla\varphi) \big) \cdot 
    (u_1 \nabla \overline{u_2} - \overline{u_2} \nabla u_1) \\
    +& e^{i\varphi} \big(A_1^2 - (A_2 + \nabla\varphi)^2 + q_1 - q_2 - (A_1 - A_2 - \nabla\varphi ) \cdot \nabla \varphi \big) u_1 \overline{u_2} \, dx \bigg| \\ 
    & \qquad \lesssim \dist(C_1, C_2) \norm{u_1}{}{H^1(\Omega)} \norm{u_2}{}{H^1(\Omega)}.
\end{align*}
\end{lemma}
\begin{proof} Since the restrictions of $ u_1 $ and $ u_2 $ to $ \Omega $ (still denoted by $ u_1 $ and $ u_2 $) satisfy $ L^\Omega_{A_1, q_1} u_1 = 0 $ and $ L^\Omega_{\overline{A_2}, \overline{q_2}} u_2 = 0 $, we have, by Proposition \ref{prop:integralID}, that
\begin{align}
    \bigg| \int_{\Omega}& i (A_1 - A_2) \cdot (u_1 \nabla \ov{u_2} -
    \ov{u_2} \nabla u_1) + (A_1^2 - A_2^2 + q_1 - q_2) u_1 \ov{u_2} \,dx \bigg| \nonumber \\ 
    &\lesssim \dist(C_1, C_2)  \label{es:gaugeOrg}
    \norm{u_1}{}{H^1(\Omega)}
    \norm{u_2}{}{H^1(\Omega)}.
\end{align}
Note that $ A_j $ and $ q_j $ have been extended as zero out of $ \Omega $, so the domain of integration of the left hand side of \eqref{es:gaugeOrg} can be trivially augmented to $ B $.

On the other hand, by identity \eqref{id:startBOUNDARY}, we know that
\begin{align}
\int_{B} i (A_1 &- A_2) \cdot (u_1 \nabla \overline{u_2} - \overline{u_2} \nabla u_1) + (A_1^2 - A_2^2 + q_1 - q_2) u_1 \overline{u_2} \, dx \nonumber \\
& = \duality{N^B_{A_1, q_1} u_1}{\overline{T_B u_2}} - \overline{\duality{N^B_{\overline{A_2}, \overline{q_2}} u_2}{\overline{T_B u_1}}} \nonumber \\
& = \duality{N^B_{A_1, q_1} u_1}{\overline{T_B u_2}} - \overline{\duality{N^B_{\overline{A_2 + \nabla \varphi}, \overline{q_2}} (e^{-i \ov{\varphi}} u_2)}{\overline{T_B (e^{- i \varphi} u_1})}}. \label{id:beforeu1u2}
\end{align}
The last identity is just a straightforward computation, which can be justified because $e^{- i \varphi} u_1 $ and $ e^{- i\overline{\varphi}} u_2 $ belong to $ H^1(B) $ and the last one satisfies
\[ L^B_{\overline{A_2 + \nabla \varphi}, \overline{q_2}} (e^{- i\overline{\varphi}} u_2) = 0 .\]
The fact that $e^{- i \varphi} u_1 $ and $ e^{- i\overline{\varphi}} u_2 $ belong to $ H^1(B) $ can be deduced, by Sobolev's embeddings, from the following inequalities
\begin{align*}
\norm{e^{- i \varphi} u_1}{}{H^1(B)} &\lesssim e^{\norm{\varphi}{}{L^\infty}}  \Big( \norm{u_1}{}{L^2} + \norm{\nabla \varphi}{}{L^n} \norm{u_1}{}{L^d} + \norm{\nabla u_1}{}{L^2} \Big) \\
\norm{e^{- i\overline{\varphi}} u_2}{}{H^1(B)} &\lesssim e^{\norm{\varphi}{}{L^\infty}}  \Big( \norm{u_2}{}{L^2} + \norm{\nabla \varphi}{}{L^n} \norm{u_2}{}{L^d} + \norm{\nabla u_2}{}{L^2} \Big)
\end{align*}
where $ d = 2n/ (n - 2) $. These estimates are consequences of H\"older inequality.

Since $\varphi$ vanishes on the boundary of $B$, we know that
\begin{equation}
T_B (e^{- i\overline{\varphi}} u_2) = T_B u_2, \qquad T_B (e^{- i \varphi} u_1) = T_B u_1 \label{id:tracesu1u2}
\end{equation}
(see for example Lemma 2 in \cite{BTo}). Thus, again by identity \eqref{id:startBOUNDARY} as well as \eqref{id:tracesu1u2}, \eqref{id:beforeu1u2} and \eqref{es:gaugeOrg} we get
\begin{align*}
    \bigg| \int_{B} &i \big(A_1 - (A_2 + \nabla\varphi) \big) \cdot 
    (u_1 \nabla (e^{i \varphi} \overline{u_2}) - e^{i\varphi} \overline{u_2} \nabla u_1) \\
    &\quad \quad + \big(A_1^2 - (A_2 + \nabla\varphi)^2 + q_1 - q_2\big) u_1 e^{i \varphi} \overline{u_2} \, dx \bigg| \\ 
    & \lesssim \dist(C_1, C_2) 
    \norm{u_1}{}{H^1(\Omega)} \norm{u_2}{}{H^1(\Omega)}.
\end{align*}
Now the integral term on last estimate can be rewritten as in the statement of the proposition.
\end{proof}

The idea will be now to use the specific Hodge decomposition of Section \ref{sec:keyPOINT} 
and write $A_1-A_2 = d\psi + \delta F$, with the fact that 
we are able control the norm of the co-exact part $\delta F$,
i.e. $\norm{A_1 - (A_2 + \nabla \psi)}{}{L^2}$
(see estimate \eqref{es:HODGEgaugeSTAB}). 
Lemma \ref{lem:gaugeINVA}, allows us then to obtain an inequality with a 
gradient term added to $A_1-A_2$. By adding $\nabla \psi$ we would thus get
an integral estimate with terms that we know how to control.
We cannot however directly add $\nabla \psi$, because
of the requirement that $\varphi|_{\p B} = 0$ in Lemma \ref{lem:gaugeINVA}.
We resolve this problem by using a cut-off argument.

We choose $\varphi$ in Lemma \ref{lem:gaugeINVA} as
$\varphi = \chi (\psi - \psi^\ast)$, where $\chi$ will be a smooth cut-off function, with
$\chi = 1$ on the supports of the potentials and such 
that it makes $\varphi$ vanish near $\partial B$ and $\psi^\ast$ is a constant. 
The idea of the cut-off argument is roughly to split $\nabla \psi$ as 
$\nabla \psi = \nabla(\chi \psi) + \nabla((1-\chi) \psi)$. Since 
$\nabla(\chi\psi) = \nabla\varphi$, this part leads to terms that can be handled with
Lemma \ref{lem:gaugeINVA}. The support of the other part $\nabla((1-\chi) \psi)$ is disjoint
from the supports of the potentials. But outside the supports of the potentials
$d\psi = \delta F$. One
can hence expect to be able to apply estimate \eqref{es:HODGEgaugeSTAB}. This is done by using the
related estimate \eqref{es:outB'STAB}.


It might be helpful for the reader to know, prior to reading Section \ref{sec:keyPOINT},
that  $\nabla \psi$ is the sum of the exact component of the Hodge decomposition of
$A_1 - A_2$ which vanishes on $\partial B$ and the exact expression of its
harmonic component.\footnote{In the more regular framework of \cite{Tz}, it was
possible to take $\nabla \psi$ as the exact component of the Hodge
decomposition of $A_1 - A_2$ vanishing on $\partial B$. Thus, there was no need
of introducing the cut-off function $\chi$ or controlling $\psi - \psi^\ast$ by
$dA_1 -dA_2$.}


\begin{proposition} \label{prop:keyPOINT} \sl There exists $\psi \in W^{1, p} (B)$ with $ p \geq 2 $ satisfying the following conditions
\begin{equation}
\norm{\psi}{}{W^{1,p}(B)} \lesssim \norm{A_1 - A_2}{}{L^p} \label{es:boundHMFdecompoSTAB}
\end{equation}
and
\begin{equation}
\norm{A_1 - (A_2 + \nabla \psi)}{}{L^2} \lesssim \norm{d(A_1 - A_2)}{}{H^{-1} \boldsymbol\Omega^2 (B)}. \label{es:HODGEgaugeSTAB}
\end{equation}

Moreover, if $B'$ is a ball containing $\ov{\Omega}$ and such that $\ov{B'} \subset B$, then
\begin{equation}
\norm{\psi - \psi^\ast}{}{H^1(B \setminus \ov{B'})} \lesssim \norm{d(A_1 -
A_2)}{}{H^{-1} \boldsymbol\Omega^2 (B)}, \label{es:outB'STAB}
\end{equation}
where $\psi^\ast$ denotes the average of $\psi$ in $B \setminus \overline{B'}$.
\end{proposition}

In order to continue with the argument, we postpone {the proof of Proposition \ref{prop:keyPOINT}, which is given in Section \ref{sec:keyPOINT}.}

In our analysis we will consider $\varphi = \chi (\psi - \psi^\ast)$ for $\chi
\in C^\infty_0 (B)$ such that $\chi (x) = 1$ for all $x \in B'$ and $ p > n $.
Thus,
\begin{equation}
e^{\norm{\varphi}{}{L^\infty}} (1 + \norm{\nabla \varphi}{}{L^n} + \norm{\nabla \psi}{}{L^n}) \lesssim 1 \label{es:boundedness1cHd}
\end{equation}
by Morrey's inequality, \eqref{es:boundHMFdecompoSTAB} and the boundedness of $B$.

For any $\xi\in\R^n$ and $ h \leq \min(1, 2/|\xi|) $, consider $ \zeta_ 1 $ and $ \zeta_2 $ as in \eqref{eq:zeta_1_2}. Let $ u_1 $ and $ u_2 $ be in the form \eqref{for:CGO1} and \eqref{for:CGO2} such that they solve $ L_{A_1, q_1} u_1 = 0 $ and $ L_{\ov{A_2}, \ov{q_2}} u_2 = 0 $ --constructed\footnote{Note that the bounded inclusion $ B^{2, r}_\varepsilon (\R^n) \hookrightarrow B^{2, \infty}_\varepsilon (\R^n) $ and \eqref{es:A_japriori2} provide an a priori bound for $ |A_j|_{B^{2, \infty}_\varepsilon} $ depending on $ M $.} as in Section \ref{sec:CGOs} for $ U = B $ and $ \tau = h^{1/(\varepsilon + 2)} $. We now state the estimate for the Fourier transform of the difference of the electric potentials by plugging in these solutions in the integral inequality given in Lemma \ref{lem:gaugeINVA}.

\begin{lemma} \label{qFourier} \sl
Let $ \theta $ belong to $ (0, 2/n) \subset \R $. There exist constants $ 0< \tilde{c} < 1 $ universal and $ c >
1 $ depending on $ \Omega, B, n $ and $ \theta $ such that
\begin{align*}
|\widehat{q_1} (\xi) - \widehat{q_2} (\xi)| \lesssim & \dist(C_1, C_2) e^{c/h}
    + \big| \log \dist(C_1,C_2) \big|^{- \tilde{c} \theta \varepsilon^2 / n} h^{-5/2}
    \\
& + h^{\varepsilon / (\varepsilon + 2)}
\end{align*}
for all $ h \leq \min(1, 2/|\xi|, h_0(cM)) $, 
provided that
\[ \big|  \log \dist(C_1,C_2) \big|^{-1} \leq 1 / c \min(1, h_0(M)).\]
Note that the implicit constant above also depends on $ \theta $.
\end{lemma}

\begin{proof}
Adding and subtracting the same terms we get that
\begin{equation}
\begin{aligned}
    \bigg| & \int_{B} e^{i \varphi} (q_1 - q_2) u_1 \ov{u_2} \, dx \bigg| 
    \leq \bigg| \int_{B} e^{i \varphi} \big(A_1^2 - (A_2+\nabla\psi)^2\big) u_1 \overline{u_2} \,dx \bigg| \\
	& + \bigg| \int_{B} i e^{i \varphi} \big(A_1 - (A_2+\nabla\psi) \big) 
        \cdot (u_1 \nabla \overline{u_2} - \overline{u_2} \nabla u_1) \,dx \bigg| \\
    & + \bigg| \int_{B} e^{i \varphi} (A_1 - A_2 - \nabla \psi) \cdot \nabla \varphi \, u_1 \overline{u_2} \,dx \bigg| + I
\end{aligned} \label{es:originq1q2}
\end{equation}
where $I$ denotes
\begin{align*}
    \bigg| \int_{B} &i e^{i \varphi} \big(A_1 - (A_2 + \nabla\psi) \big) \cdot 
    (u_1 \nabla \overline{u_2} - \overline{u_2} \nabla u_1) \\
    +& e^{i\varphi} \big(A_1^2 - (A_2 + \nabla\psi)^2 + q_1 - q_2 - (A_1 - A_2 - \nabla\psi ) \cdot \nabla \varphi \big) u_1 \overline{u_2} \, dx \bigg|.    
\end{align*}
On one hand, note that H\"older inequality and \eqref{es:boundedness1cHd} imply
\begin{align*} 
\bigg| \int_{B} i & e^{i \varphi} \big(A_1 - (A_2+\nabla\psi) \big) \cdot (u_1
    \nabla \overline{u_2} - \overline{u_2} \nabla u_1)dx \bigg| \\
& \lesssim \| A_1 - (A_2 + \nabla\psi) \|_{L^n} 
    \Big( \big \| e^{\Phi^\sharp_1} + r_1 \big \|_{L^d} \big \| \nabla(e^{\Phi^\sharp_2}+r_2) \big \|_{L^2}\\
&\quad+\sum_{j\neq k}\big \|e^{\Phi^\sharp_j}
    + r_j \big \|_{L^d} h^{-1} \big \|e^{\Phi^\sharp}_k +r_k \big \|_{L^2} \\
&  \quad + \big \| e^{\Phi^\sharp_2} + r_2 \big \|_{L^d} \big \| \nabla(e^{\Phi^\sharp_1} 
    + r_1) \big \|_{L^2} \Big)
\end{align*}
since $ \zeta_1 + \ov{\zeta_2} = h i \xi $. Here $ d = 2n/ (n - 2) $. Then, since $ B $ is bounded, estimate \eqref{es:phi_sharp}, Sobolev's embedding and \eqref{es:remainder2} imply that
\begin{equation}
\begin{aligned}
\bigg| \int_{B} i & e^{i \varphi} \big(A_1 - (A_2+\nabla\psi) \big) 
        \cdot (u_1 \nabla \overline{u_2} - \overline{u_2} \nabla u_1)dx \bigg| \\
		& \lesssim \norm{A_1 - (A_2 + \nabla\psi)}{}{L^n} h^{-(\varepsilon + 4)/(\varepsilon + 2)}.
\end{aligned} \label{es:magne1}
\end{equation}
On the other hand, H\"older inequality and \eqref{es:boundedness1cHd} imply again
\begin{align*}
\bigg| \int_{B} e^{i \varphi}  \big(A_1^2 &- (A_2+\nabla\psi)^2\big) u_1
    \overline{u_2}dx \bigg| \\
        & \lesssim \norm{A_1^2 - (A_2 + \nabla\psi)^2}{}{L^{n/2}}
    \big \|e^{\Phi^\sharp_1} + r_1\big \|_{L^d} \big \|e^{\Phi^\sharp_2} + r_2\big \|_{L^d}
\end{align*}
since $ \zeta_1 + \ov{\zeta_2} = h i \xi $. Once again, since $ B $ is bounded,
estimate \eqref{es:phi_sharp}, Sobolev's embedding, \eqref{es:remainder2},
\eqref{es:boundedness1cHd} and the a priori estimate applied to 
$\norm{A_1 + A_2 + \nabla\psi}{}{L^n}$ imply that
\begin{equation}
\begin{aligned}
\bigg| \int_{B} e^{i \varphi} \big(A_1^2 - (A_2+\nabla\psi)^2\big) &u_1 \overline{u_2}dx \bigg| \\
& \lesssim \norm{A_1 - (A_2 + \nabla\psi)}{}{L^n} h^{-4/(\varepsilon + 2)}.
\end{aligned}
\label{es:magne2}
\end{equation}
Because of the same reasons we have
\begin{equation}
\begin{aligned}
\bigg| \int_{B} e^{i \varphi} (A_1 - A_2 - \nabla \psi) \cdot \nabla \varphi \,& u_1 \overline{u_2} \,dx \bigg| \\
& \lesssim \norm{A_1 - (A_2 + \nabla\psi)}{}{L^n} h^{-4/(\varepsilon + 2)}.
\end{aligned}
\label{es:magne3}
\end{equation}
By elementary interpolation, we know that
\begin{equation*}
\|A_1 - (A_2 + \nabla\psi)\|_{L^n} 
 \leq \norm{A_1 - (A_2 + \nabla\psi)}{\theta}{L^2} \norm{A_1 - (A_2 + \nabla\psi)}{1 - \theta}{L^p},
\end{equation*}
where $ p $ is chosen to satisfies $ 1/n = \theta/2 + (1 - \theta)/p $. Note that $ p > n $. Now estimates \eqref{es:boundHMFdecompoSTAB} and \eqref{es:HODGEgaugeSTAB} imply that
\begin{equation}
\|A_1 - (A_2 + \nabla\psi)\|_{L^n}  \lesssim \norm{d(A_1 - A_2)}{\theta}{H^{-1} \boldsymbol\Omega^2 (B)}. \label{es:interMAGNE}
\end{equation}

Recall that $\varphi = \chi (\psi - \psi^\ast)$ and set $\varphi' = (1 - \chi) (\psi - \psi^\ast)$. Since $\nabla \psi = \nabla \varphi + \nabla \varphi'$ we get, by Lemma \ref{lem:gaugeINVA}, that
\begin{align*}
I &\lesssim \dist(C_1, C_2) \norm{u_1}{}{H^1(\Omega)} \norm{u_2}{}{H^1(\Omega)}  \\
&+ \bigg| \int_B i e^{i \varphi} \nabla \varphi' \cdot (u_1 \nabla \overline{u_2} - \overline{u_2} \nabla u_1) \, dx + e^{i \varphi} (\nabla \varphi  + \nabla \varphi') \cdot \nabla \varphi' u_1 \overline{u_2} \, dx \bigg|
\end{align*}
The same arguments we used to estimate \eqref{es:magne1}, \eqref{es:magne2} and \eqref{es:magne3} yield
\begin{equation}
I \lesssim \dist(C_1, C_2) \norm{u_1}{}{H^1(\Omega)} \norm{u_2}{}{H^1(\Omega)} + \norm{\nabla \varphi'}{}{L^n} h^{-(\varepsilon + 4)/(\varepsilon + 2)}
\label{es:psiLn}
\end{equation}
Note that
\begin{equation}
\norm{\nabla \varphi'}{}{L^n} \lesssim \norm{\psi - \psi^\ast}{}{W^{1,n} (B \setminus \ov{B'})}
\label{es:Ipsi}
\end{equation}
and
\begin{equation}
\begin{aligned}
\norm{\psi - \psi^\ast}{}{W^{1,n} (B \setminus \ov{B'})} &\lesssim \norm{\psi - \psi^\ast}{\theta}{H^1 (B \setminus \ov{B'})} \norm{\psi - \psi^\ast}{1 - \theta}{W^{1,p} (B)} \\
& \lesssim \norm{d(A_1 - A_2)}{\theta}{H^{-1} \boldsymbol\Omega^2 (B)}
\end{aligned} \label{es:psipsiast}
\end{equation}
by H\"older's inequality, \eqref{es:outB'STAB} and \eqref{es:boundHMFdecompoSTAB}.

Thus, \eqref{es:originq1q2}, \eqref{es:magne1}, \eqref{es:magne2},  \eqref{es:magne3}, \eqref{es:interMAGNE}, \eqref{es:psiLn}, \eqref{es:uH1}, \eqref{es:Ipsi} and \eqref{es:psipsiast} imply
\begin{align*}
    \bigg| \int_{B}  e^{i \varphi} (q_1 &- q_2) u_1 \ov{u_2} \, dx \bigg| \\
    & \lesssim \dist(C_1, C_2) e^{c/h} + \norm{d(A_1 - A_2)}{\theta}{H^{-1}
    \boldsymbol\Omega^2 (B)} h^{-(\varepsilon + 4)/(\varepsilon + 2)}.
\end{align*}
Note that this $ c $ denotes a different constant that the one in the statement.
By Proposition \ref{dAStab} we have
\begin{equation}
\begin{aligned}
    \bigg| \int_{B} e^{i \varphi} (q_1 - q_2) u_1 \ov{u_2} \, dx \bigg| &\lesssim \dist(C_1, C_2) e^{c/h} \\
    +& \big| \log \dist(C_1,C_2) \big|^{- \tilde{c}\theta \varepsilon^2 / n} h^{-(\varepsilon + 4)/(\varepsilon + 2)}.
\end{aligned} \label{es:q_1-q_2_es1}
\end{equation}
On the other hand, using \eqref{es:phi_sharp} and \eqref{es:remainder2} we see that
\begin{equation}
	\bigg| \int_{B} e^{i \varphi} (q_1 - q_2) e^{i\xi \cdot x} e^{\Phi_1^\sharp + \ov{\Phi_2^\sharp}} \, dx \bigg|    
    \lesssim \bigg| \int_{B} e^{i \varphi} (q_1 - q_2) u_1 \ov{u_2} \, dx \bigg| + h^{\varepsilon / (\varepsilon + 2)} \label{es:q_1-q_2_es2}
\end{equation}
since $ \zeta_1 + \ov{\zeta_2} = h i \xi $. Moreover,
\begin{align} \label{es:estoSEacaba}
\bigg| \int_{B} (q_1 - q_2) e^{i\xi \cdot x} \, dx \bigg|
    \lesssim \bigg| &\int_{B} (q_1 - q_2) e^{i\xi \cdot x} ( 1 - e^{\Phi_1 + \ov{\Phi_2} 
    + i \varphi}) \, dx \bigg| \nonumber\\
+ \bigg| &\int_{B} e^{i \varphi} (q_1 - q_2) e^{i\xi \cdot x} ( e^{\Phi_1 +
\ov{\Phi_2}} - e^{\Phi_1^\sharp + \ov{\Phi_2^\sharp}}) \, dx \bigg|\\
    + \bigg| &\int_{B} e^{i \varphi} (q_1 - q_2) e^{i\xi \cdot x} e^{\Phi_1^\sharp
    + \ov{\Phi_2^\sharp}} \, dx \bigg|. \nonumber
\end{align}

The first term on the right hand side of \eqref{es:estoSEacaba} can be
controlled by
\[ \norm{\Phi_1 + \ov{\Phi_2} + i \varphi}{}{L^2} \]
using \eqref{in:exponential}, \eqref{es:phi} and \eqref{es:boundedness1cHd}. Furthermore, using \eqref{eq:deltabar}, we see that
\[ (\mu_1 + i \mu_2) \cdot \nabla(\Phi_1 + \ov{\Phi_2} + i\varphi) = i (\mu_1 + i \mu_2) \cdot (A_2 + \nabla \varphi - A_1). \]
Since $\varphi$ vanishes on $\partial B$, it can be extended by zero out of $B$. Thus, by the boundedness of $((\mu_1 + i \mu_2) \cdot \nabla)^{-1}$ in weighted $L^2$ spaces in $\R^n$ (see Lemma 3.1 in \cite{SyU}), we get
\begin{align*}
\norm{\Phi_1 + \ov{\Phi_2} + i \varphi}{}{L^2} &\lesssim \norm{A_2 + \nabla
\varphi - A_1}{}{L^2} \\
&\lesssim \norm{A_2 + \nabla \psi - A_1}{}{L^2} + \norm{\psi -
\psi^\ast}{}{H^1(B \setminus \ov{B'})}\\
& \lesssim \big|  \log \dist(C_1,C_2) \big|^{-\tilde{c}\varepsilon^2 / n}.
\end{align*}
The last inequality holds because of \eqref{es:HODGEgaugeSTAB}, \eqref{es:outB'STAB} and Proposition \ref{dAStab}.

The second term  on the right hand side of \eqref{es:estoSEacaba} can be estimated by $ h^{\varepsilon / (\varepsilon + 2)} $ using \eqref{in:exponential}, \eqref{es:PhiPhiSharp}, \eqref{es:phi_sharp} and \eqref{es:phi}. Therefore,
\begin{equation}
\begin{aligned}
\bigg| \int_{B} (q_1 - q_2) e^{i\xi \cdot x} \, dx \bigg| \lesssim & \big|
    \log \dist(C_1,C_2) \big|^{-\tilde{c}\varepsilon^2 / n} + h^{\varepsilon /
    (\varepsilon + 2)} \\
& + \bigg| \int_{B} e^{i \varphi} (q_1 - q_2) e^{i\xi \cdot x} e^{\Phi_1^\sharp
+ \ov{\Phi_2^\sharp}} \, dx \bigg|.
\end{aligned} \label{es:q_1-q_2_es3}
\end{equation}
Now the result follows directly from \eqref{es:q_1-q_2_es3}, \eqref{es:q_1-q_2_es2} and \eqref{es:q_1-q_2_es1}.
\end{proof}

We next derive the stability estimate for the difference of the electric potentials in the Sobolev space $ H^{-\lambda}(\R^n) $ with $ \lambda > 0 $ using the equivalent norm given by
\[
    \|u\|^2_{H^{-\lambda}(\R^n)} = \int_{\R^n} (1+|\xi|^2)^{-\lambda} |\widehat u(\xi)|^2 \, d\xi
\]
for functions.

\begin{proposition}\label{dqStab} \sl
Consider $ \lambda \in (0, 1] $ and $ \theta \in (0, 2/n) $. There exist constants $0 < \tilde{c} < 1$ universal and $ c > 1 $ depending on $ \Omega, B, n $ and $ \theta $ such that 
\[
    \|  q_1- q_2  \|_{H^{-\lambda}(\R^n)}  
    \lesssim 
	\big|  \log \dist(C_1,C_2) \big|^{- \tilde{c} \theta \varepsilon^3 \lambda / n^2},
\]
provided that
\[ \big|  \log \dist(C_1,C_2) \big|^{-3 \tilde{c} \theta \varepsilon^2 /n}\leq 1 / c \min(1, h_0(cM)). \]
Note that the implicit constant above also depends on $\lambda $ and $ \theta $.
\end{proposition}
\begin{proof}
Let $B_\rho$ be denote the ball centred at $ 0 \in \R^n $ of radius $ \rho \geq 1 $ and let $B_\rho^c$ denote its complement in $ \R^n $. Let $q$ denote $q_1-q_2$ for clarity.
Using Lemma \ref{qFourier} we may estimate
\begin{equation*}
\begin{aligned}
    \int_{B_\rho} 
    \frac{\big|\widehat{q}(\xi)\big|^2}{(1+|\xi|^2)^\lambda} d\xi \lesssim \rho^n \big(& \dist(C_1,C_2) e^{c/h} + h^{\varepsilon/(\varepsilon + 2)} \\
    & + \big| \log \dist(C_1,C_2) \big|^{- \tilde{c}\theta \varepsilon^2 / n} h^{-5/2} \big)^2
\end{aligned}
\end{equation*}
for all $ h \leq \min(1, 2/\rho, h_0(cM)) $. Note that $ c $ and $\tilde{c}$ here are different to the ones in the statement. On the other hand,
\begin{equation*}
    \int_{B_\rho^c}
    \frac{\big|\widehat{q}(\xi) \big|^2}{(1+|\xi|^2)^\lambda} d\xi \lesssim
    \rho^{-2\lambda}
\end{equation*}
Choosing
\[ \rho = h^{- \frac{2 \varepsilon}{(2 + \varepsilon)(n + 2\lambda)}}, \]
we get
\begin{align*}
    \|  q  \|^2_{H^{-\lambda}(\R^n)} 
    \lesssim  &\dist(C_1,C_2)^2 e^{c'/h} 
    + h^{4\varepsilon \lambda / ((\varepsilon+2)(2\lambda + n))} \\
    & + \big| \log \dist(C_1,C_2) \big|^{- 2 \tilde{c} \theta \varepsilon^2 / n} h^{-6}
\end{align*}
for $ c' > 2c $. Note that this choice of $ \rho $ satisfies the restriction $ \rho \leq 2/h $.
Finally, we set
\[ h = c' \big| \log \dist(C_1, C_2) \big|^{-\tilde{c} \theta \varepsilon^2 /(6 n)} \]
to prove the statement.
\end{proof}

We next derive the stability estimate for the difference of the electric potentials in the Besov space $ B^{2, r}_0 (\R^n) $ with $ r \in [1, + \infty) $ or $ r = \infty $. We use for that the equivalent norm given by
\[ \norm{u}{r}{B^{2, r}_0 (\R^n)} = \sum_{j \in \N} \norm{\Delta_j u}{r}{L^2 (\R^n)} \]
for $ r \in [1, + \infty) $ and by
\[ \norm{u}{}{B^{2, \infty}_0 (\R^n)} = \sup_{j \in \N} \norm{\Delta_j u}{}{L^2 (\R^n)} \]
for $ r = \infty $. The family of operators $ \{ \Delta_j \}_{j \in \N} $ was described right before Proposition \ref{prop:dAbesov}.

In order to ensure the stability for the electric potentials in $ B^{2, r}_0 (\R^n) $, we will assume that $ q_j \in B^{2,r}_\varepsilon (\R^n) $ and
\[ \norm{q_j}{}{B^{2,r}_\varepsilon (\R^n)} \leq M \]
for $ j \in \{ 1, 2 \} $.
\begin{proposition} \label{prop:qSTAB} \sl
Let $ \theta $ belong to $ (0, 2/n) $ and consider $ r \in [1, +\infty) $ or $
r = \infty $. There exist constants $0 < \tilde{c} < 1$
\textcolor{red}{universal} and $ c > 1 $ depending on $ \Omega, B, n
$ and $ \theta $ such that
\[
    \|  q_1- q_2  \|_{B^{2, r}_0(\R^n)}  
    \lesssim 
	\big|  \log \dist(C_1,C_2) \big|^{-\tilde{c} \theta \varepsilon^4 / n^2},
\]
provided that
\[ \big|  \log \dist(C_1,C_2) \big|^{-36 \tilde{c} \theta \varepsilon^2 / n}\leq 1 / c \min(1, h_0(cM)) \]
Note that the implicit constant above also depends on $ \theta $.
\end{proposition}

\begin{proof} Let $q$ denote $q_1-q_2$ for clarity. Consider $ k \in \N $ to be chosen later. For any $ j \in \N $ such that $ j \leq k $ we have by Lemma \ref{qFourier} that
\begin{align*}
\norm{\Delta_j q}{}{L^2 (\R^n)} \lesssim 2^{k n/2} \big(& \dist(C_1,C_2)
    e^{c/h} + h^{\varepsilon/(\varepsilon + 2)} \\
& + \big| \log \dist(C_1,C_2) \big|^{- \tilde{c} \theta \varepsilon^2 / n} h^{-5/2} \big)
\end{align*}
for all $ h \leq \min(2^{-k}, h_0(c M)) $. Note that $ c $ and $\tilde{c}$ do not denote the ones in the statement. On the other hand, if $ j > k $, then
\[ \norm{\Delta_j q}{}{L^2 (\R^n)} \leq 2^{-k \varepsilon} 2^{j \varepsilon} \norm{\Delta_j q}{}{L^2 (\R^n)}. \]
Thus,
\begin{align*}
\norm{q}{}{B^{2, r}_{0} (\R^n)} \lesssim 2^{k n} \big(& \dist(C_1,C_2) e^{c/h}
    + h^{\varepsilon/(\varepsilon + 2)} \\
& + \big| \log \dist(C_1,C_2) \big|^{- \tilde{c} \theta \varepsilon^2 / n} h^{-5/2} \big) + 2^{-k \varepsilon}
\end{align*}
for all $ h \leq \min(2^{-k}, h_0(c M)) $. Now choosing $ k \in \N $ such that
\[ 2^{-(k+1)} < h^\frac{\varepsilon}{(\varepsilon + 2)(\varepsilon + n)} \leq 2^{-k} \]
we know that there exists a constant $ c' > 0 $ such that
\begin{align*}
\norm{q}{}{B^{2, r}_{0} (\R^n)} \lesssim & \dist(C_1,C_2) e^{c'/h} +
    h^{\varepsilon^2/((\varepsilon + 2)(\varepsilon + n))} \\
& + \big| \log \dist(C_1,C_2) \big|^{- \tilde{c} \theta \varepsilon^2 / n} h^{-3}.
\end{align*}
Note that the choice of $ k $ satisfies the restriction $ h \leq 2^{-k} $. Finally, we set
\[ h = 2 c' |\log \dist(C_1, C_2)|^{-\tilde{c} \theta \varepsilon^2 / (6 n)} \]
to prove the statement.
\end{proof}


\section{Estimating the co-exact part of the magnetic potential}\label{sec:keyPOINT}
This section is devoted to proving Proposition \ref{prop:keyPOINT},
by giving the Hodge type decomposition $A_1-A_2 = \delta F + \nabla \psi$. The proof
will be split in to two lemmas. The first lemma gives the above decomposition and the rough 
idea is to
choose the exact part $\nabla \psi $ in such a way that it is the sum
of the exact component of a Hodge decomposition of $A_1 - A_2$ which vanishes
on $\partial B$ and the exact expression of its harmonic component. The other lemma is then
devoted to showing that we can estimate the norm of the co-exact part $\delta F$, 
by the norm of $dA_1-dA_2$.

We want to point out that the decomposition given in the first lemma,
Lemma \ref{lem:HODGEdecomposition}, is slightly 
different from the 
usual Hodge-Morrey-Friedrichs decomposition in bounded domains with smooth
boundaries (see for example \cite{S}). This decomposition usually has a
harmonic component whose norm might be difficult to estimate. However, in our
case we are dealing with a domain with a straightforward topology, i.e. a ball, 
and the harmonic part can be written as an exact form and its norm can be controlled. 

Another consequence of the simple topology we are dealing with, is that
the spaces
\begin{align*}
& \mathcal{H}^1_D(B) = \{ v \in H^1\boldsymbol\Omega^1 (B): d v = 0, \delta v = 0, \mathbf{t} v = 0 \} \\
& \mathcal{H}^1_N(B) = \{ v \in H^1\boldsymbol\Omega^1 (B): d v = 0, \delta v = 0, \mathbf{n} v = 0 \}
\end{align*}
are just the trivial ones, that is, $\mathcal{H}^1_D(B) = \mathcal{H}^1_N(B) =
\{ 0 \}$. Here $\mathbf{t} v$ and $\mathbf{n} v$ denote the tangential and
normal components\footnote{More details about tangential and normal components
can be found in \cite{S}.} of $v$ on $\partial B$. Indeed, if we take $v$
either in $\mathcal{H}^1_D(B) $ or in $ \mathcal{H}^1_N(B)$, we know, by
Theorem 2.2.6 (c) or by Theorem 2.2.7 (a) in \cite{S}, that $v$ is smooth. By
Poincar\'e's lemma for closed smooth forms on contractible domains in $\R^n$, we
have that $v = dg$ with $g$ a smooth function. In consequence, we have that
$-\Delta g = \delta d g = 0$ satisfying either $\mathbf{t} dg = 0$ or
$\mathbf{n} dg = 0$, which implies that $g$ is constant in $B$. Thus $v = 0$.
The fact that $\mathcal{H}^1_D(B) = \mathcal{H}^1_N(B) = \{ 0 \}$ will be
relevant when referring to the results in \cite{S} since some arguments will
become simpler (it is here where the topology of $B$ is playing its role).

\begin{lemma} \label{lem:HODGEdecomposition} \sl Let $A_1, A_2 \in L^\infty
\boldsymbol\Omega^1 (B)$ denote the $1$-forms representing the magnetic
potentials. Then there exist $\psi \in W^{1, p} (B)$ and $F \in W^{1,p}
\boldsymbol\Omega^2 (B)$ such that
\begin{equation}
A_1 - A_2 = d\psi + \delta F, \label{for:decomposition}
\end{equation}
$\mathbf{n}F = 0$ and
 \begin{equation}
\norm{\psi}{}{W^{1,p}(B)} \lesssim \norm{A_1 - A_2}{}{L^p \boldsymbol\Omega^1 (B)} \label{es:boundHMFdecompo}
\end{equation}
for all $ p \geq 2 $. Here $\mathbf{n} F$ denote the normal component of $F$ on
$\partial B$ and $\delta$ the co-differential.

Moreover, if there exists a ball $B'$ containing $\supp A_j$ with $j \in \{ 1, 2 \}$ and such that $\ov{B'} \subset B$, then
\begin{equation*}
\norm{\psi - \psi^\ast}{}{W^{1, p}(B \setminus \ov{B'})} \lesssim \norm{\delta F}{}{L^p \boldsymbol\Omega^1 (B)},
\end{equation*}
where $\psi^\ast$ denotes the average of $\psi$ in $B \setminus \overline{B'}$.
\end{lemma}

\begin{proof} 
Through out the proof we will follow most of the notation use in \cite{S}
and we will refer to it several times. 

Let $u $ belong to $ L^p \boldsymbol\Omega^1 (B) $ with $p \geq 2$, we want to
write $u = d g + \delta f + dh $ with $\mathbf{t} g = 0$, $\mathbf{n} f = 0$
and $d h$ harmonic. Since $u$ belongs to $L^2 \boldsymbol\Omega^1 (B)$ and
$\mathcal{H}^1_D(B) = \mathcal{H}^1_N(B) = \{ 0 \}$, we have by Theorem 2.2.4
and Theorem 2.2.7 (b) in \cite{S} that:
\begin{itemize}
\item [(a)] There exists a unique $\phi_D \in H^1 \boldsymbol\Omega^1_D (B)$ (the space of forms in $H^1 \boldsymbol\Omega^1 (B)$ with vanishing tangential components on $\partial B$) such that
\[\int_B \langle d \phi_D, d v \rangle + \delta \phi_D \, \delta v  \, dx = \int_B \langle u, v \rangle \, dx\]
for any $v \in H^1 \Omega^1_D (B)$. The solution $\phi_D$ is usually called the Dirichlet potential of $u$. Note that $u$ is uniquely determined by its Dirichlet potential.

\item [(b)] There exists a unique $\phi_N \in H^1 \Omega^1_N (B)$ (the space of forms in $H^1 \Omega^1 (B)$ with vanishing normal components on $\partial B$) such that
\[\int_B \langle d \phi_N, d v \rangle + \delta \phi_N \, \delta v  \, dx = \int_B \langle u, v \rangle \, dx\]
for any $v \in H^1 \Omega^1_N (B)$. The solution $\phi_N$ is usually called the Neumann potential of $u$. Note that $u$ is uniquely determined by its Neumann potential.
\end{itemize}
By Theorem 2.2.5 (a) and Theorem 2.2.7 (b) in \cite{S}, we know that $\phi_D$ and $\phi_N$ belong to $H^2 \boldsymbol\Omega^1 (B)$. Moreover, integrating by parts in (a) and (b) above, we see that $\mathbf{t} \delta \phi_D = 0$ and $\mathbf{n} d \phi_N = 0$. Finally, by Theorem 2.2.6 (a) and Theorem 2.2.7 (c) in \cite{S}, we know that $\phi_D$ and $\phi_N$ belong to $W^{2, p} \boldsymbol\Omega^1 (B)$ for $p \geq 2$.

Define $g = \delta \phi_D$, $f = d \phi_N$ and $w = u - dg - \delta f$, then $\mathbf{t} g = 0$ , $\mathbf{n} f = 0$ and $w \in (\mathcal{E}^1(B) \oplus \mathcal{C}^1(B))^\perp$ (orthogonality in the sense of $L^2 \boldsymbol\Omega^1 (B)$) with
\begin{equation*}
 \mathcal{E}^1(B) = \{ dv : v \in H^1_0(B) \}, \quad \mathcal{C}^1(B) = \{ \delta v : v \in H^1 \boldsymbol\Omega^2_N (B) \}
\end{equation*}
and $ H^1 \boldsymbol\Omega^2_N (B)$ denoting the space of forms in $H^1
\boldsymbol\Omega^2 (B)$ with vanishing normal components on $\partial B$. The
fact that $w$ is in the orthogonal complement of $\mathcal{E}^1(B) \oplus
\mathcal{C}^1(B)$ is proven in the proof of Lemma 2.4.3 (a) in \cite{S}.
Therefore, we have that $d g \in \mathcal{E}^1(B) \cap L^p \boldsymbol\Omega^1
(B) $, $\delta f \in \mathcal{C}^1(B) \cap L^p \boldsymbol\Omega^1 (B) $ and $w
\in (\mathcal{E}^1(B) \oplus \mathcal{C}^1(B))^\perp \cap L^p
\boldsymbol\Omega^1 (B) $. Furthermore, Theorem 2.4.5 (a) in \cite{S} states
that $(\mathcal{E}^1(B) \oplus \mathcal{C}^1(B))^\perp = L^2 \mathcal{H}^1(B)$,
where $ L^2\mathcal{H}^1(B)$ is defined as the closure of
\[\mathcal{H}^1(B) = \{ v \in H^1\boldsymbol\Omega^1 (B): d v = 0, \delta v = 0 \}\]
in $L^2 \boldsymbol\Omega^1 (B)$.

We next show that $w = d h$ with $h \in W^{1, p} (B)$ (note that this will be possible because of the topology of $B$). Indeed, let $\phi$ denote the Neumann potential of $w$. By the arguments given above, we know that $\phi \in W^{2, p} \boldsymbol\Omega^1 (B) \cap H^1 \boldsymbol\Omega^1_N (B)$ and $\mathbf{n} d \phi = 0$. Defining $h = \delta \phi$ and noting that $w = \delta d \phi + d \delta \phi$, we immediately see that $w - d h = \delta d \phi $. Note that $d \phi \in H^1 \boldsymbol\Omega^2_N (B)$ and satisfies
\[\int_B \langle d(d \phi), dv \rangle + \langle \delta(d \phi), \delta v \rangle \, dx = \int_B \langle w - dh, \delta v\rangle \, dx = 0\]
for all $v \in H^1\boldsymbol\Omega^2_N (B)$ --last identity follows from a density argument together with the Green formula stated in Proposition 2.1.2 in \cite{S}. This means that $d \phi$ is a weak solution of the Hodge-Laplacian with zero Neumann boundary condition and zero right hand side. By Theorem 2.2.7 (b) in \cite{S}, we know that the unique solution for this problem in $\mathcal{H}^2_N (B)^\perp \cap H^1\boldsymbol\Omega^2_N (B) $ is the trivial one. Thus, if $d \phi \in \mathcal{H}^2_N (B)^\perp$ with
\[\mathcal{H}^2_N(B) = \{ v \in H^1\boldsymbol\Omega^2 (B): d v = 0, \delta v = 0, \mathbf{n} v = 0 \},\]
then $d \phi = 0$ and consequently $w = dh$. Finally, the fact that $d \phi $ belongs to $ \mathcal{H}^2_N (B)^\perp$ is a simple consequence of the Green formula stated in Proposition 2.1.2 in \cite{S}.

By now, we know that $u \in L^p \boldsymbol\Omega^1 (B)$ with $p\geq 2$ can be written as
\begin{equation}
u = d g + \delta f + d h \label{id:decomposition_u}
\end{equation}
with $d g \in \mathcal{E}^1(B) \cap L^p \boldsymbol\Omega^1 (B) $, $\delta f \in \mathcal{C}^1(B) \cap L^p \boldsymbol\Omega^1 (B) $ and $d h \in L^2 \mathcal{H}^1(B) \cap L^p \boldsymbol\Omega^1 (B)$. We next want to estimate $g, f$ and $h$ in terms of $u$. This will be achieved using a simple consequence of the open mapping theorem  that can be stated as follows. Let $X$ and $Y$  be two Banach spaces and let $T: X \longrightarrow Y$ be a bounded linear operator. If $T$ is bijective, then the inverse of $T$ is bounded. Let $X_D$ and $X_N$ denote the spaces
\begin{align*}
& X_D = \{ v \in W^{2, p} \boldsymbol\Omega^1 (B) \cap H^1 \boldsymbol\Omega^1_D (B) : \mathbf{t} \delta v = 0 \}, \\
& X_N = \{ v \in W^{2, p} \boldsymbol\Omega^1 (B) \cap H^1 \boldsymbol\Omega^1_N (B) : \mathbf{n} d v = 0 \};
\end{align*}
which endowed with the norm of $W^{2, p} \boldsymbol\Omega^1 (B)$ become Banach spaces. On the other hand, consider $Y_D = Y_N = L^p \boldsymbol\Omega^1 (B) $. Defining the operators $T_D$ and $T_N$ as
\[T_D v = -\Delta v \quad \forall v \in X_D, \qquad T_N v = -\Delta v \quad \forall v \in X_N; \]
we see that they are bounded and linear. Moreover, by the discussion given above about the existence, uniqueness and regularity of the Dirichlet and Neumann potentials respectively, we know that $T_D$  and $T_N$ are bijective whenever $p \geq 2$. Therefore, since $g = \delta \phi_D$, $f = d \phi_N$ and $h = \delta \phi$ with $\phi_D$ and $\phi_N$ the Dirichlet and Neumann potentials for $u$ and $\phi$ the Neumann potential for $w$, we have that
\begin{equation}
\norm{g}{}{W^{1,p}(B)} + \norm{f}{}{W^{1,p} \boldsymbol\Omega^2 (B)} + \norm{h}{}{W^{1,p}(B)} \lesssim \norm{u}{}{L^p \boldsymbol\Omega^1 (B)} \label{es:decomposition_u}.
\end{equation}
Now since $A_1 - A_2 \in L^\infty \boldsymbol\Omega^1 (B)$ the above argument can be performed for all $p \geq 2$ which provides a proof for the first part of the statement.

Finally, the second part of the statement is a simple consequence of Poincar\'e's inequality
(see \cite{E}) and the fact that $d\psi|_{B \setminus \ov{B'}} = \delta F|_{B \setminus \ov{B'}}$ (since $A_j$ is zero outside $\Omega$).
\end{proof}

We now use the properties summarized in Lemma \ref{lem:HODGEdecomposition} and
its proof to derive an estimate for the co-exact part of the decomposition
\eqref{for:decomposition}. One of the key elements of the proof is 
the Friedrich type inequality labelled as \eqref{es:H1ellipticity} below.

\begin{lemma} \label{lem:HODGEgauge} \sl Let $A_1$, $A_2$ and $\psi$ be as in Lemma \ref{lem:HODGEdecomposition} (including the conditions for $\supp A_j$). Then the following estimate holds
\begin{equation}
\norm{A_1 - (A_2 + d \psi)}{}{L^2 \boldsymbol\Omega^1 (B)} \lesssim \norm{d(A_1 - A_2)}{}{H^{-1} \boldsymbol\Omega^2 (B)}. \label{es:HODGEgauge}
\end{equation}
Moreover, if $B'$ and $\psi^\ast$ are as in Lemma \ref{lem:HODGEdecomposition}, we have that
\begin{equation}
\norm{\psi - \psi^\ast}{}{H^1(B \setminus \ov{B'})} \lesssim \norm{d(A_1 - A_2)}{}{H^{-1} \boldsymbol\Omega^2 (B)} \label{es:outB'}.
\end{equation}
\end{lemma}
\begin{proof} Note that the second part of the statement is an immediate consequence of Lemma \ref{lem:HODGEdecomposition} and \eqref{es:HODGEgauge}.

The idea to prove the first part is roughly speaking the following: $d (A_1 - A_2) = d \delta F = -\Delta F$ since $F = d \phi_N$, where $\phi_N$ is the Neumann potential of $A_1 - A_2$. Then, one should be able control $\norm{\delta F}{}{L^2 \boldsymbol\Omega^1 (B)}$ by $\norm{d(A_1 - A_2)}{}{H^{-1} \boldsymbol\Omega^2 (B)}$ since $A_1 - A_2$ has compact support inside $B'$. Let us now give a rigorous proof.

Consider a sequence $\{u_m\} \subset C^\infty_0 \boldsymbol\Omega^1 (B)$ converging to $A_1 -A_2$ in $L^2 \boldsymbol\Omega^1 (B)$ such that $\supp u_m \subset B'$ for all $m \in \N$. This is possible because $A_1 - A_2$ vanishes out of $\Omega$. Let $u_m$ be decomposed as in the proof of Lemma \ref{lem:HODGEdecomposition}, that is, $u_m = d (g_m + h_m) + \delta f_m$. Since the decomposition is orthogonal in $L^2 \boldsymbol\Omega^1 (B)$ (see Theorem 2.4.2 in \cite{S}), $\delta f_m$ converges to $\delta F$ in $L^2 \boldsymbol\Omega^1 (B)$ as $m$ goes to infinity. Since $f_m = d\phi^m_N$ with $\phi^m_N$ denoting the Neumann potential of $u_m$, we have that $f_m \in \mathcal{H}^2_N(B)^\perp$ and
\begin{equation}
\begin{aligned}
\int_B \langle d f_m , dv \rangle + \langle \delta f_m , \delta v \rangle \, dx &= \int_B \langle u_m -d(g_m + h_m), \delta v \rangle \, dx \\
& = \int_B \langle d u_m, v \rangle \, dx
\end{aligned} \label{eq:dumCOMPACTsupp}
\end{equation}
for all $v \in H^1 \boldsymbol\Omega^2_N (B)$ --both facts are consequence of the Green formula in Proposition 2.1.2 in \cite{S}. Consider now a cut-off function $\chi \in C^\infty_0 (B)$ such that $\chi (x) = 1$ for all $x \in B'$. Then,
\begin{equation}
\begin{aligned}
\left| \int_B \langle d u_m, v \rangle \, dx \right| &= \left| \int_B \langle d u_m, \chi v \rangle \, dx \right| \\
& \lesssim \norm{d u_m}{}{H^{-1} \boldsymbol\Omega^2 (B)} \norm{v}{}{H^1 \boldsymbol\Omega^2 (B)} \end{aligned}\label{es:dumCOMPACTsupp}
\end{equation}
for all $v \in H^1 \boldsymbol\Omega^2_N (B)$. On the other hand, by
Proposition 2.2.3 in \cite{S} and the fact that the Hodge star operator,
denoted by $\ast$, is an $H^1$-isometry and satisfies $\ast \mathbf{t} =
\mathbf{n} \ast $, we know that
\begin{equation}
\norm{f_m}{2}{H^1 \boldsymbol\Omega^2 (B)} \lesssim \int_B \langle d f_m , d
f_m \rangle + \langle \delta f_m , \delta f_m \rangle \, dx
\label{es:H1ellipticity}
\end{equation}
since $f_m \in \mathcal{H}^2_N(B)^\perp \cap H^1 \boldsymbol\Omega^2_N (B)$. Now by \eqref{es:H1ellipticity}, \eqref{eq:dumCOMPACTsupp}, \eqref{es:dumCOMPACTsupp} and the continuity of $d$ as operator from $L^2 \boldsymbol\Omega^1 (B)$ to $H^{-1} \boldsymbol\Omega^2 (B)$, we get
\[\norm{f_m}{}{H^1 \boldsymbol\Omega^2 (B)} \lesssim \norm{d (A_1 - A_2)}{}{H^{-1} \boldsymbol\Omega^2 (B)} + \norm{u_m - (A_1 - A_2)}{}{L^2 \boldsymbol\Omega^1 (B)}. \]
Since the second term on the right hand side of last estimate tends to vanish as $m$ grows, the sequence $	\{f_m\}$ is bounded. By a standard compactness argument, there exist a subsequence $\{f_{m_k}\}$ and $f \in H^1 \boldsymbol\Omega^2 (B)$ such that $f_{m_k}$ converges to $f$ in $L^2 \boldsymbol\Omega^2 (B)$ as $k$ goes to infinity and
\begin{equation}
\norm{f}{}{H^1 \boldsymbol\Omega^2 (B)} \lesssim \norm{d (A_1 - A_2)}{}{H^{-1} \boldsymbol\Omega^2 (B)}. \label{es:casiACABO}
\end{equation}
Finally, by the continuity of $\delta$ as operator from $L^2 \boldsymbol\Omega^2 (B)$ to $H^{-1} \boldsymbol\Omega^1 (B)$ and the uniqueness of the limit in the latter space we have that $\delta f = \delta F$. Thus, the first part of the lemma follows from \eqref{es:casiACABO} and Lemma \ref{lem:HODGEdecomposition}.
\end{proof}

\begin{remark} \sl Note that the condition of $A_1 - A_2$ having support in $\overline{\Omega} \subset B'$ is fundamental to obtain \eqref{es:HODGEgauge}. Otherwise, estimate \eqref{es:dumCOMPACTsupp} would become
\begin{equation*}
\left| \int_B \langle d u_m, v \rangle \, dx \right| \lesssim \norm{d u_m}{}{H^1 \boldsymbol\Omega^2_N (B)^\ast} \norm{v}{}{H^1 \boldsymbol\Omega^2_N (B)}
\end{equation*}
with $H^1 \boldsymbol\Omega^2_N (B)^\ast$ denoting the dual space of $H^1 \boldsymbol\Omega^2_N (B)$.
\end{remark}

\begin{acknowledgments} \rm The authors are supported by the projects ERC-2010 Advanced Grant, 267700 - InvProb and Academy of Finland (Decision number 250215, the Centre of Excellence in Inverse Problems). PC also belongs to the project MTM 2011-02568 Ministerio de Ciencia y Tecnolog\'ia de Espa\~na.
\end{acknowledgments}

\end{document}